\documentclass[12pt]{amsart}
\usepackage[utf8]{inputenc}
\usepackage{cite}
\usepackage[T1]{fontenc}
\usepackage{amsmath}
\usepackage{amsthm}
\usepackage{amsfonts}
\usepackage[english]{babel}
\usepackage{array,epsfig}
\usepackage{amssymb}
\usepackage{amsxtra}
\usepackage{amsthm}
\usepackage{latexsym}
\usepackage{dsfont}
\usepackage{mathrsfs}
\usepackage[mathscr]{eucal}
\usepackage{color}
\usepackage[all]{xy}
\usepackage{hyperref}
\usepackage{graphicx}
\usepackage{mathtools}
\usepackage{caption}
\usepackage{relsize}
\usepackage{url}
\usepackage{hyperref}

\makeatletter
\@namedef{subjclassname@2020}{%
  \textup{2020} Mathematics Subject Classification}
\makeatother

\newtheorem{thm}{Theorem}[section]

\newtheorem{lem}[thm]{Lemma}
\newtheorem{prop}[thm]{Proposition}

\theoremstyle{definition}
\newtheorem{defn}[thm]{Definition}
\newtheorem{rmk}[thm]{Remark}
\newtheorem{ex}[thm]{Example}

\newcommand{\set}[1]{\left\{#1\right\}}
\newcommand{\tuple}[1]{\left(#1\right)}

\newcommand{\abs}[1]{\left|#1\right|}
\newcommand{\norm}[1]{\left\|#1\right\|}
\newcommand{\sprod}[1]{\left<#1\right>}

\newcommand{\ol}[1]{\overline{#1}}

\newcommand{\wt}[1]{\widetilde{#1}}

\renewcommand{\phi}{\varphi}

\newcommand{\del}{\partial}
\newcommand{\delb}{\overline{\partial}}

\newcommand{\Cbb}{\mathbb{C}}

\newcommand{\Nbb}{\mathbb{N}}

\newcommand{\Pbb}{\mathbb{P}}
\newcommand{\Qbb}{\mathbb{Q}}
\newcommand{\Rbb}{\mathbb{R}}
\newcommand{\Sbb}{\mathbb{S}}

\newcommand{\Zbb}{\mathbb{Z}}

\newcommand{\Dcal}{\mathcal{D}}

\newcommand{\Fcal}{\mathcal{F}}

\newcommand{\Hcal}{\mathcal{H}}

\newcommand{\Lcal}{\mathcal{L}}
\newcommand{\Mcal}{\mathcal{M}}

\newcommand{\Ocal}{\mathcal{O}}

\newcommand{\Ucal}{\mathcal{U}}

\newcommand{\Ncal}{\mathcal{N}}

\newcommand{\Hom}{\operatorname{Hom}}

\newcommand{\Aut}{\operatorname{Aut}}

\newcommand{\Tr}{\operatorname{Tr}}

\begin{document}

\baselineskip=17pt

\title[Regularity of conical Calabi-Yau potentials]{On the regularity of conical Calabi-Yau potentials}

\author[T.-T. Nghiem]{Tran-Trung Nghiem}
\address{IMAG, Univ Montpellier, CNRS, Montpellier, France}
\email{tran-trung.nghiem@umontpellier.fr}

\begin{abstract}
Using pluripotential theory on degenerate Sasakian manifolds, we show that a locally bounded conical Calabi-Yau potential on a Fano cone is actually smooth on the regular locus. This work is motivated by a similar result obtained by R. Berman in the case where the cone is toric.  Our proof is purely pluripotential and independent of any extra symmetry imposed on the cone. 
\end{abstract}



\maketitle

\section{Introduction}
\subsection{Background and motivation}
The problem of finding Kähler-Einstein metrics has been central in the development of Kähler geometry, leading to the solution by Chen-Donaldson-Sun of the celebrated Yau-Tian-Donaldson conjecture \cite{CDSa, CDSb, CDSc}. While the problem is well understood on compact Kähler manifolds, or more generally compact Kähler varieties \cite{EGZ, Li22}, the non-compact case is still relatively open. In the pioneering work of Futaki-Ono-Wang and Martelli-Sparks-Yau \cite{FOW09}, \cite{MSY08} the existence of conical Calabi-Yau metrics (alias Ricci-flat Kähler cone metrics) on toric varieties with an isolated singularity is shown to be equivalent to a volume minimization principle for Euclidean convex cones. This principle still holds for mildly singular toric varieties as proved by Berman \cite{Ber20}. A more systematic study of polarized affine varieties with an isolated singularity was done by Collins and Székelyhidi \cite{CS19}, generalizing the work of Chen-Donaldson-Sun to the context of Kähler cones, or equivalently, Sasakian manifolds. 

A Sasakian manifold is a compact Riemannian manifold such that the metric cone over it is Kähler. Sasakian manifolds can be viewed as odd-dimensional analogs of compact Kähler manifolds since they have a natural transverse Kähler structure on an intrinsic horizontal distribution. The existence of Ricci-flat Kähler cone metrics on a Kähler cone is in fact equivalent to the existence of Sasaki-Einstein metrics on the link, which boils down to a Kähler-Einstein-like problem on the transverse structure. 

The existence of a (singular) Kähler-Einstein metrics is equivalent to solving a (degenerate) complex Monge-Ampère equation. An interesting problem to ask is the regularity of a singular Kähler-Einstein metric on the smooth locus. In the present paper, we are concerned with the regularity problem on a class of mildly singular affine varieties called \textit{Fano cones}. 

In order to state the main result, let us first give some preliminaries on Fano cones and conical Calabi-Yau potentials. Recall that a normal variety is called \( \Qbb \)-Gorenstein if a multiple of its canonical line bundle is Cartier. The action of a complex torus \( T \) on \( Y \) is said to be \textit{good} if it is effective and has a unique fixed point contained in any orbit closure. 

\begin{defn}
A \emph{cone} \( Y \) is a normal affine variety endowed with the good action of a complex torus \( T \simeq (\Cbb^{*})^k \). We say that \( Y \) is a \emph{Fano cone} if it is \( \Qbb\)-Gorenstein with klt singularities. The unique fixed point of \( Y \), denoted by \( 0_Y \), is called the \emph{vertex} of \( Y \). 
\end{defn}

Let \( \Mcal := \Hom(T, \Cbb^{*}) \simeq \Zbb^k \) be the weight lattice and \( \Ncal := \Mcal^{*} = \Hom(\Cbb^{*}, T) \) the coweight lattice. The ring of regular functions of \( Y \) admits a decomposition into \( T \)-modules 
\[ \Cbb[Y] = \oplus_{\alpha \in \Gamma} R_{\alpha}, \quad \Gamma := \set{\alpha \in \Mcal, R_{\alpha} \neq 0}, \]
where \( R_{\alpha} \) is the \( T \)-module with weight \( \alpha \). Let \( \Mcal_{\Rbb} := \Mcal \otimes \Rbb \) and \(\Ncal_{\Rbb} := \Ncal \otimes \Rbb \). The set \( \Gamma \) is an affine semi-group of finite type which generates a strictly convex polyhedral cone \( \sigma^{\vee} \subset \Mcal_{\Rbb} \). Equivalently, the dual cone \( \sigma \) in \( \Ncal_{\Rbb} \) is polyhedral of maximal dimension \( k \). This follows from the assumption that \( Y \) has a unique fixed point lying in the closure of every \(  T \)-orbit (cf. \cite{AH06}). The interior of \( \sigma \) is then non-empty and coincides with its relative interior
\[ \text{Int}(\sigma) = \set{ \xi \in \Ncal_{\Rbb}, \sprod{\alpha, \xi} > 0, \forall \alpha \in \Gamma}. \] 

\begin{defn} The interior of the cone \( \sigma \) is called the \emph{Reeb cone} of \( Y \). An element \( \xi \in \text{Int}(\sigma) \) is called a \emph{Reeb vector}. A Fano cone decorated with a Reeb vector \( (Y,\xi) \) is said to be a \emph{polarized Fano cone}. We say that \( (Y,\xi) \) is \emph{quasi-regular} if \( \xi \in \Ncal_{\Qbb} \), and otherwise \emph{irregular} if \( \xi \notin \Ncal_{\Qbb} \).
\end{defn}

The closure inside \( \Aut(Y) \) of the one-parameter subgroup generated by the infinitesimal action of \( \xi \) is a compact torus \( T_{\xi} \subset T_c \), where \( T_c \simeq (\Sbb^1)^k \) is a maximal compact subtorus of \( T \). If \( \xi \) is quasi-regular then \( T_{\xi} \simeq \Sbb^1 \), but if it is irregular then \( T_{\xi} \simeq (\Sbb^1)^m \), \( k \geq m > 1 \). Equivalently, in the quasi-regular (resp. irregular) case, the holomorphic vector field associated to \( \xi \) generates an action of \( \Cbb^{*} \) (resp. \( (\Cbb^{*})^m \)). It can be shown that in the quasi-regular case, the quotient \( (Y \backslash \set{0_Y}) / \Cbb^{*} \) is a Fano orbifold (see \cite[Paragraph 42]{Kol04}). Note however that in the irregular case, the quotient by \( (\Cbb^{*})^m\) is only well-defined as an algebraic space (cf. \cite{Kol97}). For more details on Fano cones, the reader may consult for example \cite{LLX20}, \cite{DS17} and references therein. 

Given a Fano cone \((Y,T)\), by Sumihiro's theorem (see \cite[Theorem 1, Lemma 8]{Sum74}), there exists an embedding \( Y \subset \Cbb^N \) such that \(T\) corresponds to a diagonal subgroup of \( GL(N,\Cbb)\) acting linearly. Given an embedding \( Y \subset \Cbb^N \), we say that a function \( f \) is plurisubharmonic (psh for short) on \( Y \) if it is locally the restriction to \( Y \) of a psh function on the ambient space \( \Cbb^N \).  

\begin{defn}
A \emph{\( \xi \)-radial function} (or \emph{\( \xi \)-conical potential}) \( r^2 : Y \to \Rbb_{>0} \) is a psh function on \( Y \) that is invariant under the action of \( \xi \) and 2-homogeneous under \( -J \xi \), namely 
\[ \Lcal_{\xi} r^2 = 0, \quad \Lcal_{-J\xi} r^2 = 2 r^2  \]
on \( Y_{\text{reg}} \). 
\end{defn}

If \( Y \) is a \( \Qbb \)-Gorenstein cone, then for \( m > 0 \) large enough, \( m K_Y \) is a Cartier divisor and naturally linearized by the \( T \)-action. Moreover, there exists a \( T \)-invariant non-vanishing holomorphic section \( s \in mK_Y \) and a volume form \( dV_Y \) such that 
\[ dV_Y = \tuple{i^{(n+1)^2 m } s \wedge \ol{s}}^{1/m}, \] 
where \( n+1 = \dim_{\Cbb} Y \). To simplify the notation, by an abuse of language we will sometimes say that \( s \) is a ``multivalued'' section of \( K_Y \) and simply write \( dV_Y = i^{(n+1)^2} s \wedge \ol{s} \). 

A \textit{canonical volume form} \( dV_Y \) on \( Y \) is a volume form that is \( (2n + 2) \)-homogeneous under the action of \( r \del r \), namely 
\[ \Lcal_{r \del r } dV_Y = 2 (n+1) dV_Y \]
on \( Y_{\text{reg}} \). 

The \( \Qbb \)-Gorenstein and klt singularities assumptions on \( Y \) guarantee that there exists a unique canonical volume form on \( Y \) up to a constant, see \cite{MSY08}, \cite{CS19}. 

A \( (1,1) \)-Kähler current \( \omega \) on a polarized Fano cone \( (Y,\xi) \) is said to be a \( \xi \)-\textit{Kähler cone current} if there exists a locally bounded \( \xi \)-radial function such that 
\[ \omega = dd^c r^2 \]
in the current sense.  
If moreover the function \( r^2 \) satisfies the Calabi-Yau condition 
\begin{equation} \label{conical_CY}
\omega^{n+1} = (dd^c r^2)^{n+1} = dV_Y 
\end{equation}
in the pluripotential sense of Bedford-Taylor \cite{BT76}, then \( r^2 \) is said to be a (singular) \textit{conical Calabi-Yau potential}.

\begin{defn}
We say that a Kähler cone current \( \omega = dd^c r^2 \) is a \textit{conical Calabi-Yau metric} if the function \(r^2\) is a singular conical Calabi-Yau potential which is smooth on the regular locus of \( Y \). 
\end{defn}

The motivation for studying these metrics on Fano cones actually has its origin in the compact Fano case. Concretely, Fano cones arise as metric tangent cones of the Gromov-Hausdorff limit of a Fano manifolds sequence \cite{DS17}. If each term of the sequence is moreover Kähler-Einstein, then the Fano cone admits conical Calabi-Yau metrics. As discussed in \cite[Section 4]{Ber20} (see also Remark 4.10), it is expected that a singular conical Calabi-Yau potential restricts to a smooth function on the regular locus of \( Y \).  Our goal in this article is to give an affirmative answer to this problem. 

\begin{thm}
Let \( (Y,\xi) \) be a polarized Fano cone and \( r^2 \) be a singular \( \xi \)-conical Calabi-Yau potential on \( Y \). Then \( r^2 \) is smooth on the regular locus of \( Y \). In particular, the curvature form of \( r^2 \) is a well-defined conical Calabi-Yau metric. 
\end{thm}

Such smoothness result is well-known for singular Kähler-Einstein metrics on compact Kähler varieties \cite{EGZ}, \cite{BEGZ10}, \cite[Lemma 3.6]{BBEGZ}. In the non-compact setting, when the cone has a unique singularity at the vertex, the Sasakian link is smooth, so the conical metric is automatically smooth outside the vertex. For toric Fano cones with non-isolated singularities, a regularity property was obtained by Berman \cite{Ber20} by using the toric symmetry to reformulate the problem in terms of real Monge-Ampère equations. As discussed in \cite[Remark 4.10]{Ber20}, the only places where the toric structure was used were the \( L^{\infty}\)-estimate and uniqueness of the Monge-Ampère equation. Although it is possible to generalize the same approach to a larger class of highly symmetric varieties, such as horospherical varieties, we provide a proof closer to the pluripotential spirit and independent of any symmetry other than the given effective torus action. It is an interesting problem to ask if we can weaken the regularity assumption of the solution. 

\subsection{Organization}
The organization of the article is as follows. 
\begin{itemize}
\item In Section \ref{pluripotential_sasaki}, we give a quick review of the structure of degenerate Sasakian manifolds. We then gather results in pluripotential theory on these manifolds based the on the work of Guedj-Zeriahi \cite{GZ05} and He-Li \cite{HL21}. We also introduce \textit{extremal functions} associated to a Reeb-invariant Borel set on a degenerate Sasakian manifold, which seems to be new in the literature.  These objects were not studied in \cite{HL21} in all generality (but see \cite[Prop. 3.17, Thm. 3.1]{HL21} for results concerning weighted global extremal functions). The capacity-extremal function comparison is crucial in the proof of the uniform estimate. 
\item Section \ref{section_proof_main_theorem}  is devoted to the proof of our main result. The general strategy is based on \cite{EGZ}, \cite{BEGZ10}, \cite{BBEGZ} and \cite{Ber20}. More precisely, after taking a resolution of singularities, the conical Calabi-Yau problem is translated by pullback to a Calabi-Yau problem on a degenerate Sasakian manifold. Our key result is the uniform \( L^{\infty} \)-estimate of a family of solutions, which relies on a domination-by-capacity property (cf. Prop. \ref{domination_by_capacity}). This, combined with a transverse Yau-Aubin inequality, allows us to obtain a Laplacian estimate of the family, which implies regularity of the solution.

\item In Section \ref{appendice_yau_aubin_transverse}, we provide a proof for the transverse version of Yau-Aubin inequality, which is used in the Laplacian estimate.  
\end{itemize}

\section{Pluripotential theory on Sasakian manifolds} \label{pluripotential_sasaki}

\subsection{Structure of degenerate Sasakian manifolds}
Let \( S \) be a compact differentiable manifold of dimension \( 2n + 1 \). A \textit{contact structure} on \( S \) is the data of a \( 1 \)-form \( \eta \) on \( S \) such that \( \eta \wedge (d \eta)^n \neq 0 \). The manifold \( S \) is then said to be a \textit{contact manifold}. On a contact manifold, there exists a unique vector field \( \xi \), called the \textit{Reeb vector field}, such that \( \eta(\xi) = 1, \Lcal_{\xi} \eta = 0 \). The distribution \( \Dcal := \text{ker}(\eta) \) is called the \textit{horizontal distribution} of \( S \).

In this section, we introduce the notion of \textit{degenerate Sasakian manifolds}. These manifolds were briefly mentioned in \cite{DS17} without a formal definition. Essentially, degenerate Sasakian manifolds have all the properties of a Sasakian manifold, except that the transverse \((1,1)\)-form induced by \( \eta \) is not positive-definite, hence does not define a transverse Kähler structure. Still, we assume that a degenerate Sasakian manifold has a transverse Kähler structure, but that the basic Kähler form is not induced by the contact form. 

Degenerate Sasakian manifolds arise as the link of the resolution of Fano cones (see \cite[(1), p. 367]{DS17}). The reader should compare this setting to the Kähler situation:  a resolution of a Kähler space is still Kähler, but the Kähler structure of the resolution is not the pullback of the Kähler structure on the base. 

We refer the reader to \cite{BG08} for a detailed treatment of almost contact structures and Sasakian manifolds. 
 
\begin{defn}
An \emph{almost contact structure} is given by \( (S,\xi,\eta, \Phi) \), where \( \eta \) is a contact form, \( \xi \) the corresponding Reeb vector field, and \( \Phi \) a \( (1,1) \)-tensor of \( TS \) such that
\[ \Phi^2 = -Id + \xi \otimes \eta, \quad d \eta( \Phi . , \Phi. ) = d \eta, \quad d \eta(., \Phi . ) > 0. \]  
In particular, \( \Phi|_{\Dcal} \) is an almost complex structure. 

A \emph{degenerate almost contact structure} is the same as an almost contact structure, except that \( d \eta(.,\Phi .) \) is only \emph{semipositive}. 
\end{defn}

\begin{defn}
A Riemannian metric \( g\) on a degenerate almost contact structure \( (S,\xi, \eta, \Phi) \) is said to be \emph{compatible} if 
\[g(\Phi., \Phi .) = g(.,.) - \eta \otimes \eta. \]
A \emph{degenerate metric contact structure} is a degenerate almost contact structure \( (S ,\xi, \eta, \Phi) \) endowed with a compatible metric \( g \).
\end{defn}

\begin{rmk} 
From the equalities 
\[ \Phi(\xi) = 0, \; \eta \circ \Phi = 0, \]
one can check that any compatible metric \( g \) must be of the form \( g = g_{\Dcal} \oplus \eta \otimes \eta \), where \( g_{\Dcal} \) is a metric on \( \Dcal \). In particular, if a metric \( g \) is compatible then it restricts to a Hermitian metric on \( (\Dcal, \Phi|_{\Dcal}) \). 

\end{rmk}

A (degenerate) almost contact structure is said to be \textit{normal} if the almost complex structure \(\Phi|_{\Dcal}\) is integrable (i.e. \( [\Dcal^{0,1}, \Dcal^{0,1}] \subset \Dcal^{0,1} \)).  A form \( \alpha \) on \( S \) is said to be \textit{basic} if 
\[ \Lcal_{\xi} \alpha = 0, \;  i_{\xi} \alpha = 0. \]

\begin{defn}
A \emph{degenerate Sasakian manifold} \( (S, \xi, \eta, \omega_B) \) is a \emph{normal degenerate contact structure} with a \emph{transverse Kähler metric} defined by a basic positive-definite \((1,1)\)-form \( \omega_B \). 
\end{defn}

\begin{rmk} Let \( g_B \) be the transverse Kähler metric associated to \( \omega_B\).
A degenerate Sasakian manifold admits a compatible Riemannian metric \( g_S = g_B \oplus \eta \otimes \eta \),
which restricts to the transverse Kähler metric \( g_B \) on \( \Dcal \). However, in general, \( g_B|_{\Dcal} \neq (1/2) d \eta(Id \otimes \Phi) |_{\Dcal} \) since \( d \eta(Id \otimes \Phi)|_{\Dcal} \) is only semipositive. 

In particular, a degenerate Sasakian manifold is a degenerate metric contact structure, which is generally not a classic Sasakian manifold.
\end{rmk}


\begin{ex}
Let \( Y \) be \( \Cbb^3/ \Zbb_2 \) where \( -1 \) acts as \( (z_1, z_2, z_3) \to (z_1, -z_2, -z_3) \). By direct computations of the invariant ring, one finds that 
\[ \Cbb[Y] = \Cbb[z_2 z_3, z_1, z_2^2, z_3^2] \simeq \Cbb[w_0, w_1, w_2, w_3]/ (w_0^2 - w_2 w_3), \] 
hence \( Y \) is embedded in \( \Cbb^4_{w_0,w_1,w_2,w_3} \) as the hypersurface 
\( w_0^2 = w_2 w_3 \). This is a cone with singularity along the complex line \( \set{w_0 = w_2 = w_3 = 0} \).

Let \( \xi, \eta \) be the standard Reeb vector and contact form of \(\Cbb^4\), given by
\[\xi = \sum_{j=0}^3 (y_j \del_{x_j}  - x_j \del_{y_j}), \quad \eta = \frac{\sum_{j=0}^3 y_j dx_j - x_j dy_j}{\sum_{j=0}^3 (x_j^2 + y_j^2) } \] 
The link \( L \) of \( Y \) can be identified with \( \Sbb^5 / \Zbb_2 \). In particular, with respect to the \( \Sbb^1 \)-action of \( \xi \), \( L \) is an \( \Sbb^1\)-bundle over \( \Pbb^2 / \Zbb_2 \). The latter is a Fano orbifold with a unique singularity. 

Consider the following small resolution of \( Y\)
\[ X = \set{w_0^2 + \varepsilon w_1^2 - w_2 w_3 = 0 }, \]
which after a change of coordinates can be identified with the conifold \( \sum_{j=0}^3 w_j^2 = 0 \). The link \( L' \) of \( X \) is then topologically a circle bundle over \( \Pbb^1 \times \Pbb^1 \). This is the blowup of \( \Pbb^2 / \Zbb_2 \) at the unique singularity with exceptional divisor \( E \) (cf. \cite{KW98})). 

Let \( \xi' := \pi^{*} \xi \) and \( \eta' := \pi^{*} \eta \). The restriction of \( d \eta' \) to \( L' \) vanishes on the normal directions of \( E \), and there exists naturally a transverse Kähler form \( \omega_B \) on \( L' \) coming from \( \Pbb^1 \times \Pbb^1 \). In particular, \((L', \xi'|_{L'}, \eta'|_{L'}, \omega_B )\) is a degenerate Sasakian structure, with compatible metric \( g_S = g_B \oplus \eta' \otimes \eta' \).  
\end{ex}

Many properties of Sasakian manifolds still hold on their degenerate counterparts. 
For example, on a degenerate Sasakian manifold, we still have a cover by \textit{local foliation charts}, coming from the foliation \( \Fcal_{\xi} \) by the Reeb vector field \( \xi \) on \( S \).  

\begin{defn}
The \emph{foliation atlas} on a degenerate Sasakian manifold is defined as a collection of charts \( (U_{\alpha}, \Phi_{\alpha})\) covering \( S \) with diffeomorphisms:
\begin{align*}
\Phi_{\alpha} : W_{\alpha} &\times ]-t,t[ \to U_{\alpha} \\
&(z,x) \longrightarrow (\phi_{\alpha}(z), \tau_{\alpha}(z,x)) 
\end{align*}
such that:
\begin{itemize}
\item The open interval \(]-t,t[ \subset \Rbb\) has coordinate \( x \). Here, \( t \) can be taken to be independent of \( \alpha \).
\item For all \( \alpha \), \( W_{\alpha} \simeq B_{\delta} (0) \) is the ball of radius \( \delta > 0 \) centered in \( 0 \in \Cbb^n \) with coordinates \( z = (z_1, \dots, z_n) \). Moreover, the transition map \( \phi_{\alpha \beta} := \phi_{\alpha} \circ \phi_{\beta}^{-1} \) from \( W_{\alpha} \cap W_{\beta} \) to itself is holomorphic. 
In practice, we usually take \( \delta = 1 \).
\end{itemize}
 Each chart \( (U_{\alpha}, \Phi_{\alpha}) \) is called a \emph{foliation chart}, and each \( W_{\alpha}\) is said to be a \emph{transverse chart} (or \emph{transverse neighborhood}). 

In a foliation chart \( U_{\alpha} \), we may identify \( \xi \) with \( \del_x \) and a point \( p \in S \) can be written as \( p = (z_1, \dots, z_n, x) \). 
\end{defn}

Let \( \Omega_B^k \) be the sheaf of basic \( k \)-forms on \( S \). Since the exterior differential \( d \) on \( S \) preserves basic forms, it descends to the \textit{basic exterior differential} \( d_B := d|_{\Omega_B^k} \). We then have a subcomplex \( \Omega_B^{.}(\Fcal_{\xi}) \) of the de Rham complex, and the corresponding \textit{basic cohomology} \( H_B^{*} \). The integrable complex structure on \( \Dcal \) leads to the decompositions:
\[ d_B = \del_B + \delb_B, \quad \Omega_B^k = \bigoplus_{p+q = k} \Omega_B^{p,q}, \] 
as well as the basic Dolbeault complex and the corresponding cohomologies \( H_B^{p,q} \). We then say that a basic function is \textit{transversely holomorphic} if it vanishes under \( \delb_B \). The Kähler structure on \( \Dcal \) induces the decomposition in basic cohomologies as in the classic Hodge theory:
\[ H^k_B = \bigoplus_{p+q = k} H^{p,q}_B. \] 
In short, usual Kähler properties still hold for a Kähler leaf space. We refer the reader to \cite{EKA90} for proofs. 

\subsection{Quasipsh functions and capacities}

We present here some results concerning intrinsic capacities on degenerate Sasakian manifolds, following the lines of Guedj-Zeriahi \cite{GZ05}, slightly generalizing the work of He-Li \cite{HL21}. 

Let \( (S,\xi,\eta, \omega_B) \) be a degenerate Sasakian manifold of dimension \( (2n + 1) \), where \( \omega_B \) a basic Kähler form on \( S \),  while \( \theta := d \eta \) is a smooth, semipositive and \textit{big} form; the latter meaning
\[ 0 < \text{vol}_{\theta}(S) := \int_S \theta^n \wedge \eta < +\infty. \]
Let \( g_S \) be the corresponding compatible Riemannian metric on \( S \). We denote by
\[ \mu_{\omega_B} := \omega_B^n \wedge \eta \] 
the volume form on \( S \) associated to \( g_S \).  
\begin{defn}
By a \emph{\( \xi \)-invariant object} (function, set, etc.), we mean that the object is invariant under the action of the compact torus \( T_{\xi} \) generated by \( \xi \). 

By a function in \( L^1(S) \), we mean a function being \( L^1 \) with respect to the measure \( \mu_{\omega_B} \) on \( S \). 
\end{defn}

A \textit{\((p,q)\)-transverse current} is a collection \( \set{(W_{\alpha}, T_{\alpha})} \) where \( W_{\alpha} \) is a transverse neighborhood and \( T_{\alpha} \) a current of bidegre \((p,q)\) on \(W_{\alpha}\) such that
\[ \phi_{\alpha \beta}^{*}T_{\beta}|_{W_{\alpha} \cap W_{\beta}} = T_{\alpha}|_{W_{\alpha} \cap W_{\beta}}. \]
The current \( T \) is said to be closed (resp. positive) if each \( T_{\alpha} \) is closed (resp. positive) on \(W_{\alpha}\). Recall that a basic function on \( S \) is a \( \xi \)-invariant function. A \textit{basic psh function} \(u\) on \( U_{\alpha} \) is a basic, upper-semicontinuous function on \( U_{\alpha} \) such that \( u|_{W_{\alpha}} \) is a classical psh function. In particular, \(u\) is locally integrable. 

\begin{defn}
We say that a function \(u\): \(S \to \Rbb \cup \set{-\infty} \) is \textit{basic \( \theta \)-psh} if \( u \) is locally the sum of a basic smooth function and a basic psh function, such that
\[ (\theta + d_B d_B^c u)|_{\Dcal} \geq 0 \]
in the sense of transverse currents. 

When the context is clear, we write \( d d^c \) instead of \( d_B d_B^c \). We will denote by 
\( PSH(S,\xi,\theta) \)
the set of basic \( \theta\)-psh functions. 
If \( u \in PSH(S,\xi,\theta) \), we put \( \theta_u := \theta + d d^c u \). 
\end{defn}

In particular, a \( \theta \)-psh function is \( \xi \)-invariant, upper-semicontinuous and \( L^1(S)\). A Sasakian analogue of the Bedford-Taylor theory was developed by van Coevering \cite{vC} in the case where \( \theta \) is Kähler and \( u \) is a \( \theta \)-psh bounded function on \( S \). Let us give some details of the construction. 

 Let \( u \in PSH(S,\xi,\theta) \cap L^{\infty}(S) \) and \( T \) a transverse closed positive current on \( S \). Since \( \theta \) is a closed and basic \((1,1)\)-form, \( \theta_u \) defines a transverse \((1,1)\)-current. After perharps resizing the transverse neighborhood \( W_{\alpha} \), there exists a local \( \xi \)-invariant potential \(v\) such that \( \theta = dd^c v \). We then define on each \( W_{\alpha}\),
\[ \theta_u \wedge T := dd^c( (v+u).T). \] 
This allows one to define inductively \( \theta_u^k \wedge T \) on each \( W_{\alpha} \). Passing to the foliation chart \( U_{\alpha} = W_{\alpha} \times ]-t,t[\), the Monge-Ampère operator of \( u \) is defined as
\[ \theta_u^n \wedge dx, \] 
where we identify the contact form \( \eta\) with \(dx\) in the local coordinate of \(]-t,t[\). One can check that this definition is independent of the foliation chart. We will denote the Sasakian Monge-Ampère measure of \( u\) by
\[ \text{MA}_{\theta}(u) := \theta_u^n \wedge \eta. \]
In particular, \( \text{MA}_{\theta}(u) \) is a \( \xi \)-invariant Radon measure, which has the following continuity property. 

\begin{prop} \cite[Theorem 2.3.1]{vC} \label{ma_continuity_monotone_sequences}
The Sasakian Monge-Ampère operator is continuous for monotone convergence. In other words, if \( (u_k)_{k \in \Nbb} \subset PSH(S,\xi,\theta)^{\Nbb} \cap L^{\infty}(S) \) increases (or decreases) towards \( u \), then \( \text{MA}_{\theta}(u_k) \to \text{MA}_{\theta}(u) \) in the sense of measures.  
\end{prop}

If \( u \) is bounded, then by supposing \( u \geq 0 \) and noting that \( u^2 \) is basic and psh, one can define the following transverse closed positive current
\[ du \wedge d^c u \wedge T := \frac{1}{2} dd^c u^2 \wedge T - u dd^c \wedge T. \] 
As in the (transverse) Kähler case, we have for all \( u \in PSH(S,\xi,\theta) \cap L^{\infty}(S) \),
\[ \int_S \theta_u^n \wedge \eta = \text{vol}_{\theta}(S), \] 
i.e. a locally bounded \( \theta\)-psh function is of full mass.  

We record the following regularization property for a later use.

\begin{lem} \label{regularization_theorem}
Given \( u \in PSH(S,\xi,\omega_B ) \), there exists a sequence \( (u_k)_{k \in \Nbb} \subset PSH(S,\xi,\omega_B) \cap C^{\infty}(S) \) decreasing to \( u \). 
\end{lem}

\begin{proof} 
We use the regularization procedure as in \cite[Theorem 3.3]{Ber19}. First, for a smooth basic function \( f \) and \( \beta > 0 \), consider the basic Calabi-Yau-type problem on \( S \):
\[ (\omega_B + d_B d_B^c \phi_{\beta} )^{n} \wedge \eta = e^{\beta(\phi_{\beta} - f)} \omega_B^n \wedge \eta. \] 
A solution \( \phi_{\beta} \) verifying \( \sup \phi_{\beta}  = 0 \) exists and is unique (cf. \cite[3.5.5]{EKA90}). We will denote by \( P_{\beta}(f) \), \( \beta > 0 \) the unique solution. 

Now let 
\[ P_{\omega_B}(f)(p) :=  \sup \set{\phi(p), \phi \leq f, \phi \in PSH(S,\xi,\omega_B)}. \] 
This function belongs to \( PSH(S,\xi,\omega_B) \) (cf. \cite[Proposition 3.17] {HL21}). Consider
\[P^{'}_{\omega_B}(f)(p) := \sup \set{\phi(p), \phi \leq f, \phi \in PSH(S,\xi, \omega_B) \cap C^{\infty}(S)}. \] 
Since \( u \) is u.s.c. and basic, it is a decreasing limit of a sequence of smooth basic functions \( (f_j) \). We assert that the sequence \( (v_j)_{j \in \Nbb} := (P^{'}_{\omega_B}(f_j))_{j \in \Nbb} \), which consists of basic functions, decreases to \( u \). Indeed, since \( P^{'}_{\omega_B} \) is a decreasing operator, \( (v_j) \) is a decreasing sequence and \(f_j \geq  v_j \geq u \) by construction. Since \( f_j \searrow u \), 
for all \( x\) and \( \varepsilon > 0 \), there exists \( j_0 \) such that for all \( j \geq j_0 \),
\[ u(x) \leq v_j(x) \leq f_j(x) \leq u(x) + \varepsilon, \] 
hence \( v_j(x) \) decreases to \( u(x) \). 

Arguing as in \cite[Proposition 2.3]{Ber19}, one can show that the sequence of basic \( \omega_B \)-psh functions \( v_{j,\beta} := P_{\beta}(f_j) \) converges uniformly to \( v_j \) as \( \beta \to \infty \), hence for appropriate \( \varepsilon_j \to 0 \), the sequence
\[ u_j := v_{j,\beta(j)} + \varepsilon_j, \] 
which consists of smooth basic \( \omega_B\)-psh functions, decreases to \( u \). 
\end{proof}

We also have the \textit{comparison principle} for \(\theta\)-psh functions in the degenerate Sasakian context. 

\begin{prop} \label{comparison_principle}
For all \( u, v \in PSH(S,\xi,\theta) \cap L^{\infty}(S) \), 
\[ \int_{ \set{v < u}} \text{MA}_{\theta}(u) \leq \int_{\set{v < u}} \text{MA}_{\theta}(v). \] 
\end{prop}

\begin{proof}
We first prove the following \textit{maximum principle}:
\[ 1_{\set{v < u}} \text{MA}_{\theta}(\max(u,v)) = 1_{\set{v < u}} \text{MA}_{\theta}(u). \] 
It is enough to prove the equality on a foliation chart \( U_{\alpha} \). First remark that since \( u, v \) are both basic, on \( U_{\alpha} \) they depend only on the \( z \)-coordinates, hence \( U_{\alpha} \cap \set{ v < u} = ]-t,t[ \times \set{z \in W_{\alpha}, v < u} \). Since \( \text{MA}_{\theta}(u) \) is \( \xi \)-invariant, it restricts to \( \theta_u^n \wedge dx \) on \( U_{\alpha} \). The equality is then equivalent to:
\[ 1_{ ]-t,t[ \times \set{z \in W_{\alpha}, v < u}} \theta_{\max(u,v)}^n \wedge dx = 1_{ ]-t,t[ \times \set{z \in W_{\alpha}, v < u}} \theta_u^n \wedge dx \] 
on each foliation chart. By contracting with \( \xi = \del_x \), this is exactly the classical local maximum principle for \( \theta \)-psh functions. 

It follows from the maximum principle that 
\begin{align*}
\int_{\set{v < u}} \text{MA}_{\theta}(u) &= \int_S 1_{\set{v < u}} \text{MA}_{\theta}(\max(u,v)) \\
&= \text{vol}_{\theta}(S) - \int_{ \set{v \geq u }} \text{MA}_{\theta}(\max(u,v)) \\
&\leq \int_S \text{MA}_{\theta}(v) - \int_{ \set{v > u}} \text{MA}_{\theta}(\max(u,v)) = \int_{\set{v \leq u }} \text{MA}_{\theta}(v).
\end{align*}
By arguing the same way with \( u - \varepsilon \) and \( v \), we obtain:
\[ \int_{\set{v < u - \varepsilon}} \text{MA}_{\theta}(u) \leq \int_{\set{v \leq u - \varepsilon}} \text{MA}_{\theta}(v) \leq \int_{\set{v < u}} \text{MA}(v). \] 
The proof is now concluded by remarking that \( \set{v < u - \varepsilon} \) increases to \( \set{v < u} \).  
\end{proof}

We record the following result for a later use.

\begin{prop} \label{local_dirichlet_problem}
Let \( U = B_1(0) \times ]-t,t[ \) be a foliation chart on \( S \). For every \( \phi \in PSH(S,\xi,\theta) \cap L^{\infty}(S) \), there exists a unique \( \wt{\phi} \in PSH(S,\xi,\theta) \cap L^{\infty}(S) \) such that 
\[ \text{MA}_{\theta}(\wt{\phi}) = 0 \; \text{on} \; U, \; \wt{\phi} = \phi \; \text{on} \; S \backslash U, \; \wt{\phi} \geq \phi \; \text{on} \; S. \]
Moreover, if \( \phi_1 \leq \phi_2 \), then \( \wt{\phi}_1 \leq \wt{\phi}_2 \). 
\end{prop}

\begin{proof}
The proof is a direct consequence of the local Dirichlet problem on a degenerate Sasakian manifold. The problem can be solved in exactly the same way as in the classical case by remarking that for a basic function \( u \) in a foliation chart \( (z_1, \dots, z_n,x) \). 
\[ (d_B d_B^c u)^n \wedge \eta = \det \tuple{ \frac{\del^2 u}{\del z_i \del \ol{z}_j}} \bigwedge_{k=1}^n \frac{i}{2} dz_k \wedge d \ol{z}_k \wedge dx = 0 \iff \det(u_{i \ol{j}}) = 0. \]
Hence the local Dirichlet problem on a degenerate Sasakian manifold becomes the classical Dirichlet problem (see \cite{BT76}, \cite{BT82} for a proof).
\end{proof}

\begin{prop} \label{compacity_properties}
Let \( ( \phi_j)_{j \in \Nbb} \subset PSH(S,\xi,\theta)^{\Nbb} \). 
\begin{enumerate}
    \item \label{uniform_boundedness} There exists a constant \( C = C(\mu_{\omega_B}, \theta) \) such that for all \( u \in PSH(S,\xi,\theta) \): 
    \[ - C + \sup_S u \leq \int_S u d \mu_{\omega_B} \leq \text{vol}_{\omega_B} (S) \sup_S u.  \]
    
    \item If \( ( \phi_j) \) is uniformly bounded on \( S \), then either \( (\phi_j) \) converges locally uniformly to \( -\infty \), or \( (\phi_j) \) is relatively compact in \( L^1(S) \). 
    \item \label{hartogs} If \( \phi_j \to \phi \) in \( L^1(S) \), then \( \phi \) coincides almost-everywhere with a function \( \phi^{*} \in PSH(S,\xi,\theta) \).  Moreover, 
    \[ \sup_S \phi^{*} = \lim_{j \to +\infty} \sup_S \phi_j. \]
    \item  The family 
    \[ \Fcal_0 := \set{ \phi \in PSH(S,\xi,\theta), \sup \phi = 0} \] 
    is a compact subset of \( PSH(S,\xi,\theta) \). 
\end{enumerate}
\end{prop}

\begin{proof}
 
 For \( 1) \), we can adapt the strategy in \cite[Prop. 3.3]{HL21} to the degenerate Sasakian case. Let us sketch the arguments.   
 We only need to prove the first inequality in the statement (the second one is trivial). Assuming without loss of generality that \( \sup_S u = 0\), the inequality then reduces to
 \[ \int_S u d \mu_{\omega_B} \geq -C. \] 
 There exists two finite covering of \( S \) by foliation charts \( V_{\alpha} \subset U_{\alpha} \) such that \( V_{\alpha} \simeq B_1(0) \times ]-t,t[ \) is relatively compact in \( U_{\alpha} \simeq B_4(0) \times ]-2t,2t[ \). To prove the desired result, it is enough to show that 
 \[ \int_{V_{\alpha}} u d \mu_{\omega_B} \geq - C_{\alpha}, \]
 where \( C_{\alpha} = C_{\alpha}(\theta) \). But on \( V_{\alpha} \), this is equivalent to 
 \[ \int_{B_1(0) \times ]-t,t[} u  d \mu_{z,x} = 2t \int_{B_1(0)} u(z) d \mu_z \geq - C_{\alpha}, \] 
 where \( d \mu_{z,x} \) and \( d\mu_z \) are respectively the measures \( \omega_B^n \wedge \eta \) and \( \omega_B^n \) on \( V_{\alpha} \) and \( B_1(0) \). Let \( \phi_{\alpha} \) be a local potential of \( \theta \) on \( B_4(0) \) (\( \phi_{\alpha} \) exists by the \( \del_B \delb_B \)-lemma). The function \( \phi_{\alpha} + u \) is independent of \( x \) and psh in \( B_4(0) \). By upper-semicontinuity, \( u \) attains its local supremum \( u(p_1) = 0 \)  at \( p_1 = (z_1, 0 ) \in B_4(0) \). By the submean inequality on \( B_2(z_1) \subset B_4(0) \), 
 \[ (\phi_{\alpha} + u)(z_1,0) = \phi_{\alpha}(z_1,0) \leq \frac{1}{\mu_z(B_2(z_1))} \int_{B_2(z_1)} (\phi_{\alpha} + u)(z,0) d \mu_z. \] 
 Since \( u \leq 0 \) and \( B_1(0) \subset B_2(z_1) \), this completes our proof.

 \( 2) \) is a consequence of \( 1) \) (cf. \cite[Proposition 3.4]{HL21}). 
 
 
\( 3) \) is a consequence of the local result for psh functions (see e.g. \cite[Theorem 1.46 (2)]{GZ17}). Indeed, by assumption, on each foliation chart \( U_{\alpha} \simeq B_1(0) \times ]-t,t[ \), we have \( \phi_j \to \phi \) in \( L^1_{\text{loc}}(U_{\alpha}) \). In particular, \( \phi_j \to \phi \) in \(  L^1_{\text{loc}}(B_1(0)) \) as psh functions. 

\( 4) \) is a direct consequence of \(2) \) and \( 3) \). 
\end{proof}

The following is a Chern-Levine-Nirenberg-type inequality.  

\begin{lem} \label{cln_inequality}
Let \( v, u \in PSH(S, \xi, \theta) \) such that \( 0 \leq u \leq 1 \). Then
\[ 0 \leq \int_{S} \abs{v} \theta_u^{n} \wedge \eta \leq \int_{S} \abs{v} \theta^{n} \wedge \eta  + n (1 + 2 \sup v) \text{vol}_{\theta}(S). \] 
\end{lem}

\begin{proof}
We first suppose that \( v \leq 0 \). It is enough to establish the equality for \( v_k := \max \set{v,-k} \). Indeed, the sequence \(-v_k \) increases to \( -v \), which allows us to conclude by monotone convergence theorem. Now let us prove the desired result for \( v_k \). It is clear that \( v_k \) is \( \theta \)-psh. We then have the following chain of inequalities: 
\begin{align*}
\int_{S} (-v_k) \theta_u^{n} \wedge \eta &= \int_S (-v_k) \theta_u^{n-1} \wedge ( \theta + \sqrt{-1} \del_B \delb_B u) \wedge \eta \\
&= \int_S (-v_k) \theta_u^{n-1} \wedge \theta \wedge \eta + \int_S (-v_k) \theta_u^{n-1} \wedge \sqrt{-1} \del_B \delb_B u \wedge \eta  \\
&= \int_S (-v_k) \theta_u^{n-1} \wedge \theta \wedge \eta + \int_S u \theta_u^{n-1} \wedge (- \sqrt{-1} \del_B \delb_B v_k) \wedge \eta  \\
&\leq \int_S (-v_k) \theta_u^{n-1} \wedge \theta \wedge \eta + \int_S \theta_u^{n-1} \wedge \theta \wedge \eta. 
\end{align*}
A simple induction allows us to conclude for the case \( v \leq 0 \). The general case follows by considering \( v' := v - \sup_S v \). 
\end{proof}

\begin{defn}
The capacity of a Borel set \( E \subset S \) is defined as
\[ Cap_{\theta}(E) := \sup \set{ \int_E \text{MA}_{\theta}(u), u \in PSH(S,\xi,\theta), 0 \leq u \leq 1}. \] 
\end{defn}

This definition makes sense since \( \theta \) is supposed to be big (otherwise \( Cap \) would be identically zero). It is clear by definition that \( Cap_{\theta}(.) \geq 0 \). 

Now let \( PSH^{-}(S, \xi, \theta) \) be the set of negative, basic \( \theta \)-psh functions. 

\begin{prop} \label{capacity_properties}
\hfill
\begin{itemize}
\item[1)] If \( \theta_1 \leq \theta_2 \) are two basic semipositive \((1,1)\)-forms on \( S \), then \( Cap_{\theta_1}(.) \leq Cap_{\theta_2}(.) \).  Moreover, for all \( \delta \geq 1 \), 
\[ Cap_{\theta} (.) \leq Cap_{\delta \theta} (.) \leq \delta^n Cap_{\theta}(.). \]
For every Borel set \( K \subset E \), we have:
\[ 0 \leq Cap_{\theta}(K) \leq Cap_{\theta}(E) \leq Cap_{\theta}(X) = \text{vol}_{\theta}(X). \] 
\item[2)] For all \( v \in PSH^{-}(S, \xi, \theta) \),  there exists a constant \( C = C(S, \theta) > 0 \) such that
\[ Cap_{\theta} ( v < - t) \leq \frac{C}{t}, \] 
for all \( t > 0 \). In particular, 
\(\lim_{t \to +\infty} Cap_{\theta}(v < -t) = 0 \).  

\end{itemize}
\end{prop}

\begin{proof}
1) It is clear that if \( \theta_1 \leq \theta_2 \) then \( \text{MA}_{\theta_1}(.) \leq \text{MA}_{\theta_2}(.) \) by a property of the complex Hessian in local coordinates. Moreover, if \( \theta_1 \leq \theta_2 \), then \( PSH(S, \xi, \theta_1) \subset PSH(S, \xi, \theta_2) \), so \( Cap_{\theta_1} \leq Cap_{\theta_2} \). For all \( \delta \geq 1 \) and \( u \in PSH(S,\xi, \delta \theta) \), \( 0 \leq u \leq 1 \), we have \( u \in PSH(S,\xi,\theta) \) and:
\[ 0 \leq (u/\delta) \leq (1/ \delta) \leq  1, \quad ( \delta \theta + d_B d_B^c u)^n = \delta^n \tuple{ \theta + \frac{d_B d_B^c u}{\delta} }^n. \]
Therefore  
\( Cap_{\delta \theta} (.) \leq \delta^n Cap_{\theta}(.)  \) by definition.  

For all \( K \subset E \) and all candidate function \( u \) in the definition of \( Cap \), \( \int_K \text{MA}_{\theta}(u) \leq \int_E \text{MA}_{\theta}(u) \), hence \( Cap_{\theta}(K) \leq Cap_{\theta}(E) \leq Cap_{\theta}(X) \). Finally, \( Cap_{\theta}(X) = \text{vol}_{\theta}(X) \) since a locally bounded function has full mass. 
 
2) By the Chern-Levine-Nirenberg inequality in Lem. \ref{cln_inequality}, for a \(\theta\)-psh function \( u \) such that \(  0 \leq u \leq 1 \) and \( v \in PSH(S, \xi, \theta), v \leq 0 \), we have: 
\begin{equation}
\int_S (-v) \theta_u^{n} \wedge \eta \leq \int_S (-v) \theta^{n} \wedge \eta + n \text{vol}_{\theta}(S).
\end{equation}
This inequality allows us to complete the proof. Indeed, for all \( u \in PSH(S, \xi, \theta) \) such that \( 0 \leq u \leq 1 \),
\begin{align*} 
\int_{\set{v < -t}} \theta_u^{n} \wedge \eta &\leq \frac{1}{t} \int_S (-v) \theta_u^{n} \wedge \eta \\
&\leq \frac{1}{t} \tuple{ \int_S (-v) \theta^{n} \wedge \eta + n \text{vol}_{\theta} (S) }  \\
&\leq \frac{1}{t} \tuple{ C(S, \theta) + n \text{vol}_{\theta}(S) } (\text{by Prop. \ref{compacity_properties}}).
\end{align*}
We conclude then by the definition of capacity. 
\end{proof}

The following uniqueness result still holds in the context of degenerate Sasakian manifolds. 

\begin{prop} \label{uniqueness}
Let \( u, v \in PSH(S,\xi,\theta) \cap L^{\infty}(S) \). If 
\[ \text{MA}_{\theta}(u) = \text{MA}_{\theta}(v), \] 
then \( u = v + cst \). 
\end{prop}

\begin{proof}
We borrow the proof from \cite[Theorem 3.3]{GZ07} (see also \cite[Theorem 6.4]{HL21}), which still applies when \( \theta \) is only semipositive. Let \( f = (u-v)/2 \) and \( h = (u+v)/2 \). We can assume that \( u, v \geq - C_{\theta} \) so that \( \int_{S} (-h) \theta_h^n \wedge \eta \geq 1 \). The key idea is to obtain the following inequalities:
\begin{align}
\int_S d_B f \wedge d_B^c f \wedge \theta^{n-1}_h \wedge \eta &\leq \int_S \frac{f}{2} ( \theta_u^n - \theta_v^n) \wedge \eta \label{prop_uniqueness_first},  \\
\frac{\int_S d_B f \wedge d_B^c f \wedge \theta^{n-1} \wedge \eta} { \int_{S} (-h) \theta_h^n \wedge \eta } &\leq 3^n \tuple{\int_S d_B f \wedge d_B^c f \wedge \theta_h^{n-1} \wedge \eta}^{1/2^{n-1}}. \label{prop_uniqueness_second}
\end{align}
As a consequence, if \( \theta_u^n \wedge \eta = \theta_v^n \wedge \eta \), then combining (\ref{prop_uniqueness_first}) and (\ref{prop_uniqueness_second}) yields \( \nabla f  = 0 \), hence \( u = v + cst \) as desired. 
We give a quick proof of (\ref{prop_uniqueness_first}). 

Note that the current under integration on the lhs of (\ref{prop_uniqueness_first}) is well-defined since \( u\) and \(v\) are supposed to be bounded. A direct calculation yields 
\begin{align*} 
\int_S d_B f \wedge d_B^c f \wedge \theta^{n-1}_h \wedge \eta &\leq \sum_{k=1}^{n-1} \int_S d_B f \wedge d_B^c f \wedge \theta_u^k \wedge \theta_v^{n-1-k} \wedge \eta  \\
& = \sum \int_S f (d_B d^c_B f) \wedge \theta_u^k \wedge \theta_v^{n-1-k} \wedge \eta \\
&= \int_S \frac{f}{2} ( \theta_u^n - \theta_v^n) \wedge \eta.
\end{align*}
The first inequality follows from \( C^k_{n-1} \leq 2^{n-1} \), the second one from Stokes' theorem, and the third from the fact that \( 2 d_B d_B^c f =\theta_u - \theta_v \). 

The proof of (\ref{prop_uniqueness_second}) still goes through unchanged. It consists of proving inductively that for \( T = \theta_h^l \wedge \theta^{n-2-l} \wedge \eta \), \( l = n-2, \dots, 0\), we have:
\[ \frac{\int_S df \wedge d^c f \wedge \theta \wedge T}{\tuple{\int_S (-h) \theta_h^{2} \wedge T }^{1/2}} \leq 3 \tuple{ \int_S df \wedge d^c f \wedge \theta_h \wedge T}^{1/2}, \]  
using an integration by parts and Cauchy-Schwartz inequality. 
\end{proof}

\subsection{Extremal functions}

Motivated by extremal functions in pluripotential theory, we introduce the following counterpart in the Sasakian setting. 

\begin{defn}
Let \( K \subset S \) be a  \( {\xi} \)-invariant Borel subset. The extremal function associated to  \( \theta \) and \( K \) is defined as
\[ V_{K,\theta}(p) := \sup \set{ \phi(p), \phi \in PSH(S,\xi, \theta), \phi \leq 0 \; \text{on} \; K}. \] 	
\end{defn}

Let \(V^{*}_{K,\theta} \) be the u.s.c. regularization of \( V_{K,\theta} \). We say that a \( \xi \)-invariant Borel set \( K \subset S \) is \textit{\( PSH(S,\xi,\theta) \)- pluripolar} if \( K \) belongs to the \( -\infty \) locus of a basic \( \theta \)-psh function. Clearly \( \set{u = -\infty} \) is \( \xi \)-invariant if \( u \) is basic \(\theta\)-psh.  Here we impose the symmetry by \( \xi \) on \( K \) so that there is no inherent contradiction in the definition of pluripolarity. The pluripolarity of \( K \) is determined by its extremal function, as the following lemma shows. 

\begin{lem} \label{extremal_fucntion_properties}
Let \( K \subset S \) be a \( \xi \)-invariant Borel set. 
\begin{itemize}
\item[1)] \( K \) is \( PSH(S,\xi,\theta) \)-pluripolar \(\iff  V^{*}_{K,\theta} = +\infty \iff \sup V^{*}_{K,\theta} = +\infty \).  
\item[2)] If \( K \) is not \( PSH(S,\xi,\theta) \)-pluripolar, then \( V^{*}_{K,\theta} \in PSH(S,\xi,\theta) \) and \( V^{*}_{K,\theta} = 0 \) on \( \text{Int}(K) \). Moreover, 
\[ \int_{\ol{K}} \text{MA}_{\theta} (V^{*}_{K,\theta}) = \int_{\ol{K}} (V^{*}_{K,\theta})^n \wedge \eta = \text{vol}_{\theta}(S), \quad \int_{S \backslash \ol{K}} \text{MA}_{\theta}(V^{*}_{K,\theta}) = 0. \] 
\end{itemize}
\end{lem}

\begin{proof}
1) Suppose that \( \sup_S V^{*}_{K,\theta} = +\infty \). By Choquet's lemma, there exists an increasing sequence of functions \( \phi_j \in PSH(S, \xi, \theta) \) such that \( \phi_j = 0 \) on \( K \) and \( V^{*}_{K,\theta} = (\lim \nearrow \phi_j)^{*} \). Up to extracting a subsequence, we can assume that \( \sup_S \phi_j \geq 2^j \). Define \( \psi_j := \phi_j - \sup_S \phi_j \). The sequence \( \set{\psi_j}_{j \in \Nbb} \subset PSH(S,\xi,\theta) \) is compact and satisfies \( \int_S \psi_j d \mu_{\omega_B} \geq - C(\mu_{\omega_B} ) \) (cf. Lem. \ref{compacity_properties}). Let 
\[ \psi := \sum_{j \geq 1} 2^{-j} \psi_j. \] 
The function \( \psi \) is basic \(\theta\)-psh as a limit of basic \(\theta\)-psh functions, and satisfies \( \int \psi d \mu_{\omega_B} \geq -C(\mu_{\omega_B}) \). It is clear that \( \psi_j(x) = - \sup_S \phi_j \), \( \forall x \in K \), hence \( K \subset \set{\psi = - \infty} \). 

Now suppose that \( K \subset \set{\psi = -\infty} \) where \( \psi \in PSH(S,\xi,\theta) \). For all \( c \in \Rbb \), \( \psi +c \in PSH(S, \xi,\theta) \) and \( \psi +c \leq 0 \) on \( K \). It follows that \( V_{K,\theta}^{*} \geq \psi +c \), hence \( V_{K,\theta}^{*} = + \infty \) on \( S \backslash \set{\psi = -\infty} \). Finally, \( V_{K,\theta}^{*} =  +\infty \) on \( S \) since \( \set{\psi = -\infty} \) has zero mass with respect to \( \mu_{\omega_B} = \omega_B^n \wedge \eta \). 

2) Clearly \( V_{K,\theta}^{*}  = 0 \) in \( \text{Int}(K) \) by definition. The function \( V_{K,\theta} \) is basic as the sup-envelope of basic functions, hence its u.s.c. regularization \( V^{*}_{K,\theta} \) is also basic. The fact that \( V^{*}_{K,\theta} \) is \( \theta \)-psh follows from (\ref{hartogs}) of Prop. \ref{compacity_properties}. 
Since a locally bounded \( \theta \)-psh function has full mass, we have:
\[ \int_{\ol{K}} \text{MA}_{\theta}(V^{*}_{K,\theta}) = \int_S \text{MA}_{\theta}(V^{*}_{K,\theta}) = \int_{S} (\theta + d_B d_B^c V^{*}_{K,\theta})^n \wedge \eta = \text{vol}_{\theta}(S). \] 
It only remains to show that
\( \text{MA}_{\theta}(V^{*}_{K,\theta}) = 0 \) on \( S \backslash \ol{K} \), which is equivalent to showing
\[ \int_{U_{\alpha}} \text{MA}_{\theta}(V^{*}_{K,\theta}) =  0 \] 
on each foliation chart \( U_{\alpha} = B_1(0) \times ]-t,t[ \subset S \backslash \ol{K} \). 
By Choquet's lemma, there exists an increasing sequence of functions \( \phi_j \in PSH(S, \xi, \theta) \) such that \( \phi_j = 0 \) on  \( K \) and \( V^{*}_{K,\theta} = (\lim \nearrow \phi_j)^{*} \). Let \( \wt{\phi}_j \) the unique solution of local Dirichlet problem with initial datum \( \phi_j \) (which exists by Prop. \ref{local_dirichlet_problem}). In particular, 
\[ \text{MA}_{\theta}(\wt{\phi}_j) = 0 \; \text{on} \; U_{\alpha}. \] 
Moreover, the sequence \( (\wt{\phi}_j) \) is increasing and \( \wt{\phi}_j = \phi_j \) on \( S \backslash U_{\alpha} \), hence \( \wt{\phi}_j = 0 \) on \( K \). This shows that \( \wt{\phi}_j \leq V_{K,\theta}^{*} \) , therefore \( \wt{\phi}_j \nearrow V_{K,\theta}^{*} \).  
By continuity of the Monge-Ampère operator along a monotone sequence (cf. Thm \ref{ma_continuity_monotone_sequences}), \( \text{MA}_{\theta}(V^{*}_{K,\theta}) = 0 \) on \( U_{\alpha} \). 
\end{proof}

Let us now state an important comparison theorem between capacity and extremal functions. 

\begin{lem} \label{capacity_extremal_comparison}
Let \( M_{K,\theta} := \sup_S V^{*}_{K,\theta} \). For all compact non-pluripolar and \( \xi \)-invariant \( K \subset S \) we have:
\[ 1 \leq \text{vol}_{\theta}(S)^{1/n} Cap_{\theta}(K)^{-1/n} \leq \max(1, M_{K,\theta}). \] 
\end{lem}

\begin{proof}
The inequality on the left is clear by Prop. \ref{capacity_properties}. For the inequality on the right, we will consider two cases. First, suppose that \( M_{K,\theta} \leq 1 \), then \( V_{K,\theta}^{*} \) is bounded. Since \( K \) is non-pluripolar,  \( V^{*}_{K,\theta} \in PSH(S, \xi,\theta) \). Moreover, \( \text{MA}_{\theta} (V^{*}_{K,\theta}) \) is supported in \( K \) (cf. Lem. \ref{extremal_fucntion_properties}), hence
\[ Cap_{\theta}(K) \geq \int_{K} \text{MA}_{\theta} (V^{*}_{K,\theta}) = \int_S \text{MA}_{\theta}(V^{*}_{K,\theta}) = \text{vol}_{\theta}(S) , \]
which completes the proof in the \( M_{K,\theta} \leq 1 \) case. 

Assume now that \( M := M_{K,\theta} \geq 1 \). Since the function \( V_{K,\theta}^{*}  / M \) is a candidate in the definition of \( Cap_{\theta} \), it follows that 
\begin{align*} 
Cap_{\theta}(K) &\geq \int_K \text{MA}_{\theta}(M^{-1} V^{*}_{K,\theta} ) \\
&= \int_S \text{MA}_{\theta} (M^{-1} V^{*}_{K,\theta}) \; (\text{by Lem. \ref{extremal_fucntion_properties}}) \\
&\geq M^{-n}\int_S \text{MA}_{\theta}(V^{*}_{K,\theta}) = M^{-n} \text{vol}_{\theta}(S).
\end{align*} 
This allows us to conclude. 
\end{proof}

\subsection{Lelong number and integrability}

We define the \textit{Lelong number} of a basic psh function \( u \) on a foliation chart \( U_{\alpha} \) at a point \( p \) with coordinates \( (z,x) \) by 
\[ \nu(u,p) := \lim_{r \to 0^{+}} \frac{1}{\log(r) \text{vol}(B(z,r))} \int_{B(z,r)} u(z) \omega_B^n. \]   

This number does not depend on the foliation chart since the transition maps restrict to biholomorphisms on transverse neighborhoods and that the right-hand side is invariant under biholomorphisms by a theorem of Siu. 

It is clear by our definition that the Lelong number is \( \xi \)-invariant. Moreover, in a foliation chart \( B_1(0) \times ]-t,t[ \), the function \( x \in ]-t,t[ \to \nu(u,(z,x)) \) is constant for all \( z \in B_1(0) \). The Lelong number at a point \( p \) on a Sasakian manifold therefore equals its value at the projection of \( p \) to the transverse holomorphic ball of a foliation chart. Local properties of Lelong number can be translated word by word to the Sasakian setting. 

\begin{prop}
The number 
\[ \nu(\set{\theta}) := \sup \set{ \nu(\phi,x), (\phi,x) \in PSH(S,\xi,\theta) \times S } \] 
is finite and depends only on the basic cohomology class of \( \theta \). 
\end{prop}

\begin{proof}
Since \( S \) is compact, there exists a basic Kähler form \( \theta' \) such that \( \theta' \geq \theta \), hence \( PSH(S,\xi,\theta') \supset PSH(S,\xi,\theta) \), so \( \nu(\set{\theta'}) \geq \nu(\set{\theta}) \). It is then enough to prove the assertion when \( \theta \) is transverse Kähler.  

For \( p \in S \), we define \( \chi \) to be a smooth function equals to \( 1 \) in a neighborhood of \( p \) and \( 0 \) outside a larger neighborhood. Let
\[ g_p(.) := \chi(.) \log d(.,p), \] 
where \( d \) is the Riemannian distance associated to \( \theta \).
It is clear that \( g_p \) is smooth on \( S \backslash \set{p} \) and psh on a neighborhood of \( p \), hence \( A \theta \)-psh for \( A > 0 \). Since \( S \) is compact, we can choose a uniform constant \( A = A(\theta) \) such that for all \( p \in S \),
\[dd^c g_p \geq -A \theta. \] 
By taking average with respect to the action of the compact torus generated by \( \xi \), we can suppose that \( g_p \) is \( \xi \)-invariant, hence \( g_p \in PSH(S,\xi, A \theta) \). 

A basic psh function \( \phi \) in a foliation chart \( B_1(0) \times ]-t,t[ \) restricts to a psh function on the ball \( B_1(0) \), so we have 
\[\nu(\phi,0) = \int_{\set{0_z}} d_B d_B^c \phi \wedge (d_B d_B^c \log \abs{z})^{n-1} \]
with \( 0_z \) being the center of \( B_1(0) \) (see e.g. \cite[Lemma 2.46]{GZ17} for a proof). 
It follows from this local result that for any \( a = (z,x) \),
\[
\nu(\phi,a) = \int_{\set{z}} \theta_{\phi} \wedge (A \theta + d_B d_B^c g_a)^{n-1}.
\]
The right-hand side is bounded by \( \int_S A^n \theta^n \wedge \eta = A^n \text{vol}_{\theta}(S) \). This completes our proof. 
\end{proof}

\begin{thm} \label{skoda_integrability}
Let \( \Fcal_0 := \set{ \phi \in PSH(S,\xi,\theta), \sup_S \phi  = 0} \). If 
\[ A < 2 \nu(\set{\theta})^{-1}, \] 
then 
\[ \sup_{ \phi \in \Fcal_0} \set{\int_S e^{-A \phi} \omega_B^n \wedge \eta} \leq C, \] 
for a constant \( C \) depending only on \( \omega_B \) and \( \theta \). 
\end{thm}

\begin{proof}
We will reduce the problem to the classic Skoda's integrability theorem. First remark that there exist two covers of \( S \) by a finite number of foliations charts \( (V_j)_{1 \leq j \leq N} \) and \( (U_j)_{1 \leq j \leq N} \), where \( U_j = B_{1}(0) \times ]-t,t[ \), such that \( \ol{V}_j \subset U_j \). We need to show that on each foliation chart \( U_j \), there exists a constant \( C_j = C(V_j, \Fcal_0, A) \) satisfying
\[ \int_{U_j} e^{-A \phi} \omega_B^n \wedge \eta \leq  C_j. \] 
But since on \( U_j \), \( \phi \) depends only on the \( z \) coordinates and \( \eta \) coincides with \( dx \), it is enough to show that 
\[ \int_{U_j} e^{-A \phi} \omega_B^n \wedge \eta = 2 t \int_{B_{1}(0)} e^{- A \phi } \omega_B^n \leq C_j. \]
This follows from the local Skoda's integrability theorem since the family \( \Fcal_0 \) is compact (cf. \cite[Theorem 2.50]{GZ17} for a proof).  
\end{proof}

\section{Regularity of the potential} \label{section_proof_main_theorem}

This part is dedicated to the proof of our main theorem. Let us first give some preliminaries and outline the arguments of the proof. Consider a Fano cone \( Y \) of complex dimension \( n + 1 \) with a good action by \( T \simeq (\Cbb^{*})^k \). Let \( T_c \simeq (\Sbb^1)^k \) be the maximal compact subtorus of \( T \). 

Consider a \( T \)-equivariant embedding of \( Y \) into \( \Cbb^N \) such that \( T \) corresponds to a diagonal group of \( GL(N,\Cbb) \) with its standard action on \( \Cbb^N \). Recall that \( \xi \) generates the action of a compact torus \( T_{\xi} \subset T_c \). Now fix a locally bounded conical Calabi-Yau potential \(r^2 \) (which exists by assumption) and a Reeb vector \( \xi \) on \( Y \), whose action by \(T_{\xi}\) extends to \( \Cbb^N \) through the embedding. By \cite{HS16}, there exists a radial function \( r_{\xi}^2 \) on \( \Cbb^N \) associated to \( \xi \), which defines a conical metric \(\omega_{\xi} = dd^c r^2_{\xi} \). The function \( r_{\xi}^2 \) restricts to a \(\xi\)-conical potential on \( Y \). The link of \( Y \) is homeomorphic to the set \( Y \cap \set{r_{\xi}^2 = 1} \). Now let
\[\pi: X \to Y \]
be a \( T \)-equivariant resolution of \( Y \) (which exists by Lem. \ref{resolution_singularities}).  Let 
\[ \Ucal := \pi^{-1} (Y_{\text{reg}})  \] 
be the open Zariski subset of \( X \) isomorphic to \( Y_{\text{reg}} \).  
Consider the following submanifold of \( X \):
\[(S = \pi^{-1}( Y \cap \set{r_{\xi}^2 = 1} ), \xi, \eta, \omega_B), \]
where by an abuse of notation \(\xi \) still denotes the pullback of the given Reeb field on \( \Cbb^N \), \( \omega_B \) is a transverse Kähler form  on \( S \) (cf. Lem. \ref{basic_form_asymptotic}),  and \( \eta = 2 \pi^{*} d^c \log r_{\xi}^2 \) the contact form on \( S \), which is pullback of the contact form associated to \( \xi \) on \( \Cbb^N \). Since \( d \eta \) is only semipositive, \( S \) is degenerate Sasakian. 
One can show (see Prop. \ref{conical_eqn_equivalent_transverse_eqn}) that the conical Calabi-Yau equation
\[ (dd^c r )^{n+1} = dV_Y \] 
is in fact equivalent to the following transverse equation on \( \Ucal \cap S \): 
\[ (\theta_X + d_B d_B^c \phi_X)^{n} \wedge \eta = e^{-(n+1) \phi_X } e^{(n+1)(\Psi_{+} - \Psi_{-})} \omega_B^{n} \wedge \eta. \]
Here,
\begin{itemize}
\item \( \theta_X := d \eta \),
\item \( \phi_X := \pi^{*} \phi, \; \phi := \log (r^2 / r_{\xi}^2) \), 
\item \( \Psi_{\pm} \) are basic \( A \omega_B \)-quasi-psh on \( S \) for \( A > 0 \) large enough.
\end{itemize}
Remark that \( \theta_X \) is a semipostive and big form on \( S \). By construction, \( \phi_X \) is invariant under the induced actions of \( \xi \) and \( - J \xi \) on \( X \). In a foliation chart \( (z_1, \dots, z_n,x) \) of \( S \), the equation can be written as: 
\[ \det \tuple{ \theta_{X,i\ol{j}} + \frac{\del^2 \phi_X }{\del z_i \del \ol{z}_j }} = e^{-(n+1) \phi_X(z) } e^{(n+1)(\Psi_{+}(z) - \Psi_{-}(z))} \det( \omega_{B, i \ol{j}} ). \]  
 The smoothness of \( r^2 = r_{\xi}^2 e^{\phi}  \) on \( Y_{\text{reg}} \) is then equivalent to the regularity of \( \phi_X := \pi^{*}  \phi \) on \( S \cap \Ucal \). Consider the family of equations:
\[ (\theta_X + \varepsilon \omega_B + d_B d_B^c \phi_{j, \varepsilon} )^{n} = e^{(n+1)(\psi_{+,j} - \psi_{-,j})} \omega_B^{n}, \] 
where \( \psi_{\pm,j} \) are two sequences of basic \( A \omega_B \)-qpsh functions decreasing to \( \psi_{+} := \Psi_{+} \) and \( \psi_{-} := \Psi_{-} + \phi_X \) for \( A > 0 \) large enough. The existence of a unique \( \phi_{j, \varepsilon} \) verifying \( \sup \phi_{j,\varepsilon} = 0 \) is guaranteed by the transverse Calabi-Yau theorem of \cite{EKA90}.  Finally, to obtain the regularity of \( \phi_X \), we proceed by the following classic steps:
\begin{itemize}
\item[1)] \textit{Uniform estimate:} The functions \(\phi_{j,\varepsilon}\) are uniformly bounded, i.e. there exists a constant \( C \) independent of \( j \) and \( \varepsilon \), such that
\[ \norm{\phi_{j, \varepsilon}}_{L^{\infty}(S)} \leq C.  \] 
\item[2)] \textit{Laplacian uniform estimate:} Using the uniform estimate of the previous step, one can show that there exists \( C' \) such that for all \(j, \varepsilon \), 
\[ \sup_{S \cap \Ucal} \abs{\Delta_{\omega_B} \phi_{j,\varepsilon}} \leq C', \] 
where
\[ \Tr_{\omega_B} f := n \frac{d_B d_B^c f \wedge \omega_B^{n-1} }{\omega_B^{n} }. \] 
\item[3)] By the complex Evans-Krylov theory, we obtain the following uniform estimate:
\[ \norm{ \phi_{j,\varepsilon}}_{C^{2,\beta}(S)} \leq C'', \] 
which implies \( C^{k+2, \beta} \)-estimates for all \( k > 0 \) by Schauder estimate and a bootstrapping argument. 
\end{itemize}

The last step is classic and well-known in the literature (cf. \cite{Blo05}). Our focus will be mostly on the first and second steps (see Prop. \ref{linfty_estimate_pluripotential} and Prop. \ref{laplacian_estimate}).  

\subsection{Transverse Kähler form}

Let \( V \) be an irreducible projective variety. Following \cite[Paragraph 3]{Kol07}, by a \textit{strong resolution} we mean a proper morphism \( \pi : V' \to V \) such that 
\begin{itemize}
    \item \( V' \) is smooth and \( \pi \) is birational,
    \item \( \pi: \pi^{-1}(V_{\text{reg}}) \to V_{\text{reg}} \) is a biholomorphism,
    \item \( \pi^{-1}(V_{\text{sing}}) \) is a divisor with simple normal crossings (s.n.c). 
\end{itemize}
In the sense of \cite[Paragraph 4]{Kol07}, we say that a resolution \( \pi_V : V' \to V \)  is \textit{functorial} if every smooth morphism \( \phi: V \to W \) can be lifted to a smooth morphism \( \phi': V' \to W' \) such that \( \pi_W \circ \phi' = \phi \circ \pi_V \), where \( \pi_W : W' \to W \) is a resolution of \( W \). 

\begin{lem} \label{resolution_singularities}
There exists a \( T \)-equivariant resolution of singularities \( \pi: X \to Y \). 
\end{lem}

\begin{proof}
Let us embed \( Y \) in a \( T \)-equivariant manner into \( \Cbb^N \) such that \( T \) is identified with a diagonal group. Let \( \ol{Y} \subset \Pbb^N \) be the closure of \( Y \) in \( \Pbb^N \), then one can find a \( T \)-equivariant resolution \( \pi : \ol{X} \to \ol{Y} \). Indeed, it is enough to take  \( \pi \) as a strong and functorial resolution in the sense of Kollar as recalled above (see \cite[Theorem 36]{Kol07} for a proof of existence).

The functoriality of the resolution implies that the action of all algebraic group on \( \ol{Y} \) lifts on \( \ol{X} \) in such a way that \( \pi \) is equivariant (see \cite[Paragraph 9]{Kol07}). We conclude that \(\pi:  X := \ol{X} \cap \Cbb^N \to Y \) is a \( T \)-equivariant resolution of \( Y \). 
\end{proof}

Now let \( (X, \pi) \) be the resolution of  \( Y \), constructed in the previous lemma. Let $E_0 := \pi^{-1}(0_Y)$ be the ``vertex exceptional divisor''. Since \( \pi \) is equivariant, the vector fields $\xi$ and $ - J \xi$ induce by pullback the respective actions on $X$ (still denoted by \( \xi \) and \( -J \xi \)). The action generated by \( - J \xi \) is an action of \( \Rbb^{*}_{+} \). 

The pullback by \( \pi \) of the holomorphic vector field \( v_{\xi} := (-J \xi - \sqrt{-1} \xi)/2 \)  defines a holomorphic foliation \( \Fcal_{v_{\xi}} \) on \( X \backslash E_0 \). At every point \( p \in X \backslash E_0 \), there exist \textit{transverse holomorphic coordinates} \((z_1,\dots,z_n, w) \) such that
\[ v_{\xi} . z_j = 0, \; \frac{\del}{\del \Im w } = \xi, \; \frac{\del}{\del \Re w} = (- J \xi), \] 
which restrict to the foliation coordinates \( (z,x) \) on \( S \). In other words, \( w = \pi^{*} \log r_{\xi} + \sqrt{-1} x \). A form \( \alpha \) on \( X \backslash E_0 \) is said to be \textit{basic} if 
\[ \Lcal_V \alpha = 0, \; i_V \alpha = 0, \; \forall V \in \Rbb \set{\xi, -J\xi}. \] 
The restriction map allows us to identify basic forms on \( X \backslash E_0 \) and basic forms on \( S \). 

\begin{lem} \label{initial_kahler_form}
There exists a  $T_c$-invariant Kähler form $\omega$ on $X$ and a global smooth function $\Phi_{\omega}$ defined on $U$ such that 
\begin{equation*}
dd^c \Phi_{\omega} = \omega, \; \Phi_{\omega} \to -\infty \; \text{near} \; \del \Ucal.
\end{equation*}
\end{lem}

\begin{proof}
Let \( \pi: \ol{X} \to \ol{Y} \subset \Pbb^N \) be the resolution as in the previous lemma. Let \( \Ocal(1) \) be the \( T \)-linearized hyperplane line bundle  of \( \Pbb^N \). 


Since \( \Ocal(1) \) is ample,  \( \pi^{*}(\Ocal(1)) \) is big and nef. It then follows by a well-known property (see e.g. \cite[Theorem 1.4.13]{DFEM}) that one can find a \( T \)-equivariant effective Cartier divisor \( F \) and a \( T \)-equivariant ample \( \Qbb\)-line bundle \( A \) on \( \ol{X} \) such that  
\[ \pi^{*}\Ocal(1) = A + F . \] 
Now let \( \pi': \ol{X}' \to \ol{X} \) be a \( T \)-equivariant log resolution of the pair \( (\ol{X}, F) \). We can then proceed as in \cite[Lemma 2.2.9]{DFEM} to show that there exist positive numbers \( b_j > 0 \), and an ample \( \Qbb\)-line bundle \( A' \) on \( \ol{X}' \) such that
\begin{equation*}
(\pi')^{*} \pi^{*} \Ocal(1) = A' + F', \; F' := \sum b_j E'_j,
\end{equation*}
where \( E_j' \) are the components of the exceptional divisor of \( \pi \circ \pi': \ol{X}' \to \ol{Y} \). Without loss of generality, we can suppose that \( \ol{X}' = \ol{X}, A' = A \) and \( F' = F = \sum b_j E_j \) where the \( E_j \) are the components of the exceptional divisor of \( \pi: \ol{X} \to \ol{Y} \).  

Now let \( \norm{.}_{F} \) be a \( T_c \)-invariant metric on \( F = \set{s_{F} = 0} \) and \( \phi_{F} := -\log \norm{s_{F}}^2_{F} \) its potential. Let \( h_A \) be a \( T_c \)-invariant metric of strictly positive curvature on \( A \) and \( h \) the \( T_c \)-invariant metric \( h := h_A e^{-\phi_{F}} \) on \( \pi^{*} \Ocal(1) \).
 
 Since \( X \) is contained in an open affine set \( \simeq \Cbb^N \) of \( \Pbb^N \), there exists a global trivializing \( T_c \)-invariant section \( \pi^{*} s_{\Cbb^N} = (s_A \otimes s_F)|_X \) of the line bundle \( \pi^{*} \Ocal(1)|_X \), induced by the trivialization \( s_{\Cbb^N} \) of \( \Ocal(1)\) over \( \Cbb^N\). The global function \( \Phi_{\omega} := - \log h_A(s_A) |_X = - \log h(\pi^{*} s_{\Cbb^N}) - \phi_F \lvert_{X} \) is smooth over \( \Ucal \) and its curvature defines a Kähler form (by ampleness of \( A \))
\[ \omega := dd^c \Phi_{\omega}. \]
Finally, remark that \( s_F \to 0 \) near \( \del \Ucal \), so we have \( \Phi_{\omega} = - \log h - \phi_F \to -\infty \) near \( \del \Ucal \).     
\end{proof}

\begin{lem} \label{basic_form_asymptotic} \cite[Prop. 4.3]{Ber20} 
There exists a global smooth function $\Phi_B$ on $\Ucal$ satisfying
\[ \Lcal_{\xi} \Phi_B = 0, \Lcal_{-J \xi} \Phi_B = 2, \Phi_{B} \to -\infty \; \text{near} \; \del \Ucal   \]
and a transverse basic Kähler form \( \omega_B \) on \( X \backslash E_0 \) such that \( dd^c \Phi_B = \omega_B \) on \( \Ucal \). 
\end{lem}

\begin{rmk} 
The information on the behavior of \( \Phi_B \) near the border of  \( \Ucal \) is crucial in the Laplacian estimate of the potential \( \phi_X \).  
\end{rmk}

\begin{proof}
The proof in \cite{Ber20} is an adaptation of the construction of reduced Kähler metrics on a symplectic quotient (see e.g. \cite[Formulae 4.5, 4.6]{BG04}). We provide here the details for the reader's convenience. 

Remark however that in our case, the symplectic quotient is not well defined since the action generated by \( \xi \) on the level set of the hamiltonian is not free in general. However, the construction still applies since it is local in nature.    

Let \( \omega \) be the  \( T_c \)-invariant Kähler form on \( X \), constructed in Lem.  \ref{initial_kahler_form}. Remark that the action generated by \( \xi \) is hamiltonian with respect to \( \omega\) (since by the embedding of \( Y \) into \( \Cbb^N \), \( \xi \) is identified with a hamiltonian action on \( \Cbb^N \)). It follows that there exists a smooth function \( \Hcal: X \to \Rbb \) such that
\[ d \Hcal (.) = - \omega( \xi, .) = g_{\omega} (- J \xi, .) \]
where \( g_{\omega} \) is the metric associated to \( \omega \). 
In particular, \( d \Hcal (- J \xi) > 0 \), so \( d_x \Hcal \) is surjective for \( x \notin E_0 \). It follows that \( \Hcal \) is a submersion for \( x \notin E_0\);  hence for  \( \lambda \) positive, sufficiently large, 
\[ S_{\lambda} = \set{ \Hcal = \lambda} \] 
is a compact submanifold of \( X \backslash E_0 \), diffeomorphic to \( (X \backslash E_0)/ \Rbb^{*}_{+} \). Now let 
\[ \pi_{\lambda} : X \backslash E_0 \to S_{\lambda}, \quad i_{\lambda} : S_{\lambda} \to X \backslash E_0 \] 
be the natural projection and inclusion. Let \( \Phi_{\omega} \) be the global potential on \( \Ucal \) constructed in Lemma \ref{initial_kahler_form}. Let \( V_p \) be the neighborhood of a point \( p \in X \backslash E_0 \) with local transverse coordinates \( (z,w) \). Consider the  following \( \xi \)-invariant function on \( S_{\lambda} \cap \Ucal \): 
\[ \Phi_{\lambda} := i^{*}_{\lambda} (\Phi_{\omega} - \lambda \Im w ). \] 
The function 
\begin{equation} \label{local_basic_potential_equation}
\Psi_B = \pi^{*}_{\lambda} \Phi_{\lambda} + \lambda \Im w = \pi^{*}_{\lambda} \Phi_{\omega}|_{S_{\lambda}} + \lambda(\Im w - i^{*}_{\lambda} \Im w )
\end{equation} 
is then \( \xi \)-invariant on \( V_p \)  and well-defined on \( V_p	\). Indeed, let \( V_{p'} \) be another local transverse neighborhood of a point \( p' \in S_{\lambda} \cap V_p \). By the definition of \( w \), \( v_{\xi} (w-w') = 0 \), so there exists a basic transversely holomorphic function \( f(z) \) on \( V_p \cap V_{p'} \) such that \( w-w' = f(z) \). It follows that
\[ \Im (w - w')|_{V_p \cap V_{p'}} = \Im(w-w')|_{S_{\lambda} \cap V_p \cap V_{p'}} = i^{*}_{\lambda} \Im(w-w') |_{V_p \cap V_{p'}}. \] 
By construction, we have
\( \Lcal_{-J \xi} \Psi_B = \lambda\), 
hence \( \Psi_B \) extends uniquely to a smooth function on \( \Ucal \). The function 
\[ \Phi_B := 2 (\Psi_B/ \lambda ) \] 
satisfies \( \Lcal_{\xi} \Phi_B = 0, \; \Lcal_{-J \xi} \Phi_B = 2 \). We assert that the following global form on  \( \Ucal \) 
\[ \omega_B := dd^c \Phi_B \] 
defines a transverse Kähler metric on \( \Ucal \). By a direct computation from the equation (\ref{local_basic_potential_equation}) as in \cite[Section 9]{BG04}, \( 2 \lambda^{-1} \omega \) is exactly \( \omega_B \) on \( S_{\lambda} \). After replacing \( \Phi_B \) with \( 2 \lambda^{-1} \Phi_{\omega} \) on each \( V_p \), we see that \( \omega_B \) extends to a transverse Kähler metric on \( X \backslash E_0 \). 

It remains to show that \( \Phi_B \to -\infty \) on \( \del \Ucal \). Indeed, on \( \Ucal \cap V_p \), 
\( \Phi_B - 2 \lambda^{-1} \Phi_{\omega} = \Im w - i^{*}_{\lambda}(\Im w) \) for all \( p \in S_{\lambda} \). It follows that \( \Phi_B - 2 \lambda^{-1} \Phi_{\omega}\) is bounded on \( S_{\lambda} \cap \Ucal \), so \( \Phi_B = (\Phi_B -2 \lambda^{-1} \Phi_{\omega}) + 2 \lambda^{-1} \Phi_{\omega} \to -\infty \) near \( \del \Ucal \)  since \( \Phi_{\omega} \to -\infty \) near \( \del \Ucal \). 
\end{proof}

Since \( X \) is a \(T_c \)-invariant resolution of \( Y \) and that \( Y \) has klt singularities, there exists a \( T_c \)-invariant divisor \( D \) such that
 \[ \pi^{*} K_Y = K_X + D, \; D =  \sum_{a_j > -1} a_j D_j. \]
 We have moreover a decomposition $D = D_{+} - D_{-}$, where
 \[ D_{+} := \sum_{a_j > 0} D_j, \; D_{-} := \sum_{a_j < 0}(-a_j) D_j \] 
 are two effective \( T_c \)-invariant \(\Qbb\)-divisors. 
 There exist then a \( T_c  \)-invariant volume form \( dV_X \) on \( X \), two multivalued sections \( s_{\pm} \) and hermitian \( T_c \)-invariant metrics \( h^{\pm} \) on \(D_{\pm} \), such that
 \begin{equation} \label{volume_discrepancy}
 \pi^{*} dV_Y = \norm{s_{+}}^2_{h^{+}} \norm{s_{-}}^{-2}_{h^{-}} dV_X. 
 \end{equation} 
To be precise, we may choose
\[ \norm{s_{+}}^2_{h^{+}} := \prod_{a_j > 0} \abs{s_j}^{2a_j}_{h_j}, \quad \norm{s_{-}}^2_{h^{-}} := \prod_{a_j < 0 } \abs{s_j}^{-2a_j}_{h_j}, \]
where \( h_j \) are \( T_c \)-invariant hermitian metrics of the fiber \( \Ocal_{X} (D_j) \).

\begin{lem} \label{volume_pullback}
There exist two basic quasi-psh \( T_c \)-invariant functions \( \Psi_{\pm} \) on \( S \), smooth on \( \Ucal \) and a constant \( A > 0 \) such that on \( S \), 
\[ \pi^{*} dV_Y ( -J \xi, .) = e^{(n+1)(\Psi_{+} - \Psi_{-} )} \omega_B^{n} \wedge \eta, \quad \frac{i}{2 \pi} \del_B \delb_B \Psi_{\pm} \geq -A \omega_B. \]
Moreover, \( e^{-\Psi_{-}}  \in L^p(S), p > 1 \).  
\end{lem}

\begin{proof}
On a local foliation chart \( U_0 \subset X \), one can find a basic function \( v \in L^{\infty}(U_0) \cap C^{\infty}(\Ucal \cap U_0) \) such that the following equality between volume form on \( S \) holds  
\[ dV_X (  -J \xi, . ) = e^{ v} \omega_B^{n} \wedge \eta = e^{v_{+} - v_{-}} \omega_B^n \wedge \eta, \]
where \( e^{v_{+}} \) (resp. \( e^{v_{-}} \)) are functions on \( U_0 \) that vanish on \( D_{+} \) (resp. \( D_{-} \)) and non-zero elsewhere. This equality is moreover independent of the foliation chart.  
Assume that there exists a positive constant \( C > 0 \) satisfying
\begin{equation} \label{section_norm_estimate}\frac{i}{2 \pi } \del_B \delb_B \log \norm{s_{\pm}}^2_{h^{\pm}|_S} \geq - C \omega_B.
\end{equation}
Then by choosing local functions \( \Psi_{\pm} \) on \( U_0 \) as
\[
 (n+1) \Psi_{\pm} :=  \log \norm{s_{\pm}}^2_{h^{\pm}|_S} + v_{\pm}|_{S},
\] 
we obtain from (\ref{volume_discrepancy})
\begin{align*}
\pi^{*} dV_Y(-J \xi, .) &= e^{ \log \norm{s_{+}}^2_{h^{+}} + v_{+} - \log \norm{s_{-}}^2_{h^{-}} - v_{-} } \omega_B^{n} \wedge \eta \\
&= e^{(n+1)(\Psi_{+} - \Psi_{-})} \omega_B^n \wedge \eta. 
\end{align*}
 and the estimate of \( \del_B \delb_B \Psi_{\pm} \) follows immediately. 

It remains to prove (\ref{section_norm_estimate}). By definition of \( s_{\pm} \) and \( \norm{.}_{h^{\pm}} \), in a transverse holomorphic chart of \( X \backslash E_0 \) with coordinates \( (z,w) \),   there exist \( T_c \)-invariant local potentials \( \phi_{\pm} \) and holomorphic \( T_c \)-semi-invariant local functions \( f_{\pm} \) such that
\[ \norm{s_{\pm}}_{h^{\pm}} = \abs{f_{\pm}(z,w)} e^{-\phi_{\pm}(z,w)}. \] 
In particular, there exist \( \lambda_{\pm} \in \Rbb \) satisfying
\[ \frac{\del}{\del \Im w} f_{\pm} = i \lambda_{\pm} f. \]  
After replacing \( f_{\pm} \) by \( f_{\pm} e^{-\lambda_{\pm} w} \), one can suppose that \( f_{\pm} \) are \( \xi \)-invariant (hence basic), so \( \delb_B f_{\pm} = 0 \). It follows that \( f_{\pm} \) are transversely holomorphic, hence \( d_B d^c_B \log \abs{f_{\pm}(z,w)}^2 \geq 0 \), so locally,
\[ d_B d^c_B \log \norm{s_{\pm}}^2_{h^{\pm} |_S} \geq - C d_B d_B^c \phi_{\pm}, \]
for some constant \( C \) depending only on the local open set. Moreover, since \( \omega_B \) is Kähler, one can find in a transverse neighborhood a constant \( A > 0 \) (which depends only on the neighborhood) such that
\[ d_B d^c_B \phi_{\pm} \leq A \omega_B. \] 
The compactness of \( S \) then completes the proof of (\ref{section_norm_estimate}). Finally, since \( Y \) has klt singularities, \( D_j \) are normal crossing divisors, hence there exists \( p > 1 \) such that \( pa_j > -1 \) for all \( j \), so \( e^{-\Psi_{-}} \in L^p(S) \) for some \( p > 1 \).
\end{proof}

\subsection{Transverse Monge-Ampère equation}

\begin{prop} \label{conical_eqn_equivalent_transverse_eqn}
The conical potential \( r \) is a solution in the pluripotential sense of the equation
\begin{equation}
(dd^c r ^2)^{n+1} = dV_Y 
\end{equation}
on \( Y_{\text{reg}} \) if and only if \( \phi_X \) satisfies the following equation on \( S \cap \Ucal \):  
\begin{equation} \label{transversal_CY_eqn}
(\theta_X + d_B d_B^c \phi_X)^{n} \wedge \eta = e^{-(n+1) \phi_X} e^{(n+1)(\Psi_{+} - \Psi_{-})}  \omega_B^{n} \wedge \eta.
\end{equation}
In particular, in a transverse holomorphic neighborhood \( S \cap \Ucal \), 
\[ (\theta_X + d_B d^c_B \phi_X)^n = e^{-(n +1)\phi_X} e^{(n+1)(\Psi_{+}  - \Psi_{-} )} \omega_B^{n}. \]
\end{prop}

\begin{proof}
By definition \( \Phi = \log r ^2 \), hence
\[ dd^c r^2 = e^{\Phi} (dd^c \Phi + d \Phi \wedge d^c \Phi ) = r^2 (dd^c \Phi + d \Phi \wedge d^c \Phi )   \] 
in the current sense.  
We have 
\[ ( dd^c \Phi + d \Phi \wedge d^c \Phi)^{n+1} = \sum c_{k,n} (dd^c \Phi)^k \wedge ( d \Phi \wedge d^c \Phi)^{n-k} = ( dd^c \Phi)^{n} \wedge d \Phi \wedge d^c \Phi. \]
Indeed, in the transverse coordinates \( (z,w) \) on \( X \backslash E_0 \),  
\[ \frac{\del \Phi }{\del w} = \frac{\del \Phi}{\del \ol{w} } = 1, \] 
hence \( (dd^c \Phi)^{n+1} = 0 \). 
It follows that 
\begin{align*} 
(dd^c r^2)^{n+1} = dV_Y  \iff r^{2n + 2} (dd^c \Phi)^{n} \wedge d \Phi \wedge d^c \Phi  =  dV_Y.   
\end{align*}
Since 
\[ \Lcal_{\xi} \Phi = 0, \]
the restriction of \( \Phi \) in \( S \) is basic.
It follows that 
\begin{align*} 
(dd^c \Phi)^{n} \wedge d \Phi \wedge d^c \Phi &=  \det \tuple{ \frac{\del^2 \Phi} {\del z_l \del \ol{z}_m}} \bigwedge (i/2) dz_k \wedge d \ol{z}_k \wedge d \Phi \wedge d^c \Phi \\
& = (d d^c \Phi)^{n} \wedge (dw +  d \ol{w})   \wedge  (d^c w + d^c \ol{w}) \\ 
& =  (d d^c \Phi)^{n}  \wedge 2 d \Re w \wedge 2 d^c \Re w. 
\end{align*}
The conical Calabi-Yau equation then becomes 
\begin{equation*}
r^{2n + 2} (dd^c \Phi)^{n} \wedge 2 d \Re w \wedge 2 d^c \Re w =  dV_Y. 
\end{equation*}
By contracting the equality with \(- J \xi \), and using \( 2 d \Re w(-J\xi) = 1 \), we have: 
\[ r^{2n + 2} (dd^c \Phi)^{n} \wedge  2 d^c \Re w = dV_Y(-J \xi), \]
then using \( dd^c \Phi = \theta + dd^c \phi = \theta + d_B d_B^c \phi \), we obtain on \( Y \)
\[ r^{2n+2} (\theta + d_B d_B^c \phi)^n \wedge 2 d^c \Re w = dV_Y(-J \xi,.). \]  
Next, by pulling back the equation and using \( 2 \pi^{*} d^c \Re w = \eta \), together with Lemma \ref{volume_pullback}, we obtain the following equations on  \( \Ucal = \pi^{-1}(Y_{\text{reg}}) \)
\begin{align*}
&\quad (\pi^{*} r^{2n+2}) (\pi^{*} \theta + d_B d_B^c \pi^{*} \phi)^n \wedge \eta = \pi^{*} dV_Y(-J\xi,.) \\
&\iff (\pi^{*} r_{\xi}^{2n+2}) e^{(n+1) \pi^{*} \phi} (\theta_X + d_B d_B^c \phi_X)^n \wedge \eta = e^{(n+1)(\Psi_{+} - \Psi_{-})} \omega_B^n \wedge \eta \\
&\iff (\pi^{*}r_{\xi}^{2n+2}) (\theta_X + d_B d^c_B \phi_X)^n \wedge \eta = e^{-(n+1) \phi_X} e^{(n+1)(\Psi_{+} - \Psi_{-})} \omega_B^n \wedge \eta.
\end{align*}
It follows that on \( S \cap \Ucal = \Ucal \cap \pi^{-1}(\set{r_{\xi}^2 = 1}) \), one has
\[ (\theta_X + d_B d^c_B \phi_X )^{n} \wedge \eta  = e^{-(n +1)\phi_X} e^{(n+1)(\Psi_{+}  - \Psi_{-} )} \omega_B^{n} \wedge \eta. \] 
Finally, by applying \( i_{\xi} \) and using that \( \eta(\xi) = 1 \), the equation on \( S \cap \Ucal \)  becomes 
\[ (\theta_X + d_B d^c_B \phi_X)^n = e^{-(n +1)\phi_X} e^{(n+1)(\Psi_{+}  - \Psi_{-} )} \omega_B^{n}. \]
The converse is proved in the same manner. 
\end{proof}

\subsection{Uniform estimate}

Let \( \psi_{\pm,j} \) be two sequences of smooth basic quasi-psh functions which decrease to  
\[ \psi_{+} := \Psi_{+}, \quad \psi_{-} := \Psi_{-} + \phi_X,\]
and such that
\begin{equation} \label{regularization}
\quad d_B d^c_B \psi_{\pm,j } \geq - C \omega_B 
\end{equation}
for a uniform constant \( C \) independent of \( j \). Such a sequence exists by virtue of Lem.  \ref{regularization_theorem}. 

Let \( \varepsilon > 0 \). Recall that the form \( \theta_X = \pi^{*} dd^c \log r_{\xi}^2 \) is semi-positive, big and basic, hence \( \theta_X + \varepsilon \omega_B \) is a transverse Kähler form. Consider the following equation on \( S \) for a smooth basic \( (\theta_X + \varepsilon \omega_B) \)-psh function \( \phi_{j,\varepsilon} \): 
\begin{equation} \label{transversal_CY_eqn_perturbed}
\tuple{\theta_X + \varepsilon \omega_B + d_B d^c_B \phi_{j, \varepsilon}}^{n} \wedge \eta  = e^{(n+1)( \psi_{+,j} - \psi_{-,j}) } \omega_B^{n} \wedge \eta.    
\end{equation}
By the transverse Calabi-Yau theorem of El-Kacimi Alaoui \cite[3.5.5]{EKA90}, for all  \(j, \varepsilon \), there exists a unique basic solution satisfying
\[\sup \phi_{j, \varepsilon} = 0.\]

Now let \( \mu_j \) be the smooth volume form \( e^{(n+1)(\psi_{+,j} - \psi_{-,j})} \omega_B^n \wedge \eta \) on \( S \). The following lemma is elementary:

\begin{lem}
Let \( \mu \) be an inner-regular positive Borel measure on \( S \). Then for all \( \xi \)-invariant Borel set \( E \subset S \), 
\[ \mu(E) = \sup \set{\mu(K), K \subset E \; \text{compact}, \xi-\text{invariant}}. \] 
In particular, \( \mu_j \) satisfies this property.  
\end{lem}

\begin{proof}
It is enough to show that for all \( j \in \Nbb^{*} \), there exists a compact \( \xi \)-invariant \( K_j \) such that: 
\[ \mu(E) \leq \mu(K_j) + \frac{1}{j}. \] 
By inner regularity of  \( E \), there exists a compact \( C_j \subset E \) such that:
\[ \mu(E) \leq \mu(C_j) + 1/j. \] 
The idea is to average \( C_j \) by the action of \( T_{\xi} \). We define 
\[ K_j := \cup_{ g \in T_{\xi} } g. C_j = T_{\xi} . C_j. \] 
For each \( j \), the set \( K_j \) is compact and \( \xi \)-invariant by construction.  Moreover, \( K_j \subset E \) since \( g.C_j \subset g.E \subset E \). Finally, the fact that \( C_j \subset K_j\) implies \( \mu(C_j) \leq \mu(K_j) \). This completes our proof. 
\end{proof}

We also have the important \textit{domination by capacity} property of the measures \( \mu_j \).

\begin{prop} \label{domination_by_capacity}
The measures \( \mu_{j} \) satisfy the \( \Hcal(\alpha, A, \theta) \) condition for all \( \alpha \). Namely, for all \( \alpha > 0 \), there exists a constant \( A \) independent of \(j\) such that
\[ \mu_j (E) \leq A Cap_{\theta}(E)^{1+\alpha}, \] 
for all \(\xi\)-invariant Borel subset \( E \subset S \). 
\end{prop}

\begin{proof}
By inner regularity of \( \mu_j \), it is enough to establish the lemma for a compact \( \xi \)-invariant \( K \subset  S \). Indeed, suppose that the inequality is true for all such \( K \), then for all Borel \( \xi \)-invariant set \( E \), 
\begin{align*}
 \mu_j(E) &= \sup \set{ \mu_j(K), K \subset E \; \text{compact}, \xi-\text{invariant}  }\\
 &\leq A \sup \set{ Cap_{\theta}(K)^{1+\alpha}, K \subset E \; \text{compact}, \xi-\text{invariant} } \\
 &\leq A Cap_{\theta}(E)^{1+\alpha} \; (\text{by Prop. \ref{capacity_properties}(1))}.
\end{align*}

We can suppose furthermore that \( K \) is non-pluripolar (otherwise \( \mu_j(K) = 0 \) and the inequality is then trivial). 

Now let \( K \) be a compact \( {\xi} \)-invariant and non-pluripolar. Let \( p > 1 \) be as in Lemma \ref{volume_pullback}. By Hölder inequality, 
\[ 0 \leq \mu_j(K) \leq \norm{f_j}_{L^p(\omega_B^n \wedge \eta )} \text{vol}_{\omega_B }(K)^{1/q}, \]
where \( 1/p + 1/q = 1 \). Since \( \psi_{+,j} \leq \psi_{+,1} \) and \( \psi_j \geq \psi_{-} \), the function \( e^{(n+1)(\psi_{+,j} - \psi_{-,j})} \) is bounded in \( L^p \) by \( e^{(n+1)(C - \psi_{-})} \), where \( C := \sup_S \psi_{+,1} \). It follows that the norm \( \norm{f_j}_{L^p(\omega_B^n \wedge \eta)} \) is uniformly bounded, therefore it is enough to show that
\[ \text{vol}_{\omega_B}(K) \leq C \exp \tuple{-\gamma (Cap_{\theta}(K))^{-1/n}}, \]
where \( C = C(\theta, \omega_B) , \gamma = \gamma(\theta) \) are constants independent of \( j \). The conclusion then follows from the elementary equality \( \exp(-x^{\beta}) \leq A_{\alpha} x^{\alpha} \), for all \( x \in [0,1], \alpha > 0 \).  

By Theorem \ref{skoda_integrability}, for \( \gamma := 2 / (\nu(\set{\theta}) + 1 ) \), there exists a constant \( C = C(\theta, \omega_B) \) such that
\[ \sup_{\psi \in \Fcal_0} \int_S \exp( - \gamma \psi) \omega_B^n \wedge \eta \leq C. \]       
In particular, for \( \psi := V^{*}_{K,\theta} - M_{K,\theta} \) (recall that \( M_{K,\theta} = \sup V^{*}_{K,\theta} \)), we obtain
\[ \int_{S} \exp(- \gamma V^{*}_{K,\theta} ) \omega_B^n \wedge \eta \leq C \exp(-\gamma M_{K,\theta} ). \] 
Note that \( V^{*}_{K,\theta} \) is well defined thanks to the \( \xi \)-invariance of \( K \). Finally, since \( V^{*}_{K,\theta} \leq 0 \) \( \mu_{\omega_B} \)-a.e. on \( K \), we have
\[ \text{vol}_{\omega_B}(K) \leq C \exp(-\gamma M_{K,\theta} ). \] 
An application of Lemma  \ref{capacity_extremal_comparison} then completes our proof. 
\end{proof}

Let us first establish some more useful lemmas before proving the uniform estimate. 

\begin{lem} \label{capacity_estimate}
Let \(u \in PSH(S, \xi, \theta) \cap L^{\infty}(S) \) be a negative function. For all \( s \geq 0 \), \( 0 \leq t \leq 1 \), 
\[ t^n Cap_{\theta} ( u < -s - t) \leq  \int_{\set{u < -s}} \theta_u^{n} \wedge \eta. \]  
\end{lem}

\begin{proof}
Let \( v \in PSH(S,\xi,\theta) \), \( 0 \leq v \leq 1 \). Then 
\[ \set{ u < - s -t} \subset \set{u \leq tv-s-t} \subset \set{u < -s}. \] 
By definition of the Monge-Ampère operator 
\[ \int_{\set{u < -s-t}} \text{MA}_{\theta} (v) \leq \int_{\set{u \leq tv - s - t}} \text{MA}_{\theta}(v) \leq t^{-n} \int_{\set{u \leq tv - s - t}} \text{MA}_{\theta} (tv). \] 
Applying the comparison princple \ref{comparison_principle} to the functions \( u + s + t \) and \( tv \),
\[ t^{-n} \int_{\set{u \leq tv - s - t}}  \text{MA}_{\theta} (tv) \leq t^{-n} \int_{\set{u \leq tv - s - t}} \text{MA}_{\theta}(u) \leq \int_{\set{u < -s}} \text{MA}_{\theta}(u), \]  
which terminates our proof. 
\end{proof}

\begin{lem} \cite[Lem. 2.4]{EGZ} \label{vanishing_capacity}
Let \( f : \Rbb^{+} \to \Rbb^{+} \) be a right-continuous decreasing function such that \( \lim_{s \to +\infty} f(s) = 0 \). If \( f \) satisfies the condition 
\[H(\alpha,B),  \quad t f(s+t) \leq Bf(s)^{1 + \alpha}, \; \forall s \geq 0, 0 \leq t \leq  1, \]
then there exists \( s_0 = s_0(\alpha, B) \) such that \( f(s) = 0 \), \( \forall s \geq s_0 \). 
\end{lem}

\begin{prop} \label{linfty_estimate_pluripotential}
There exists a uniform constant \( C \) such that
\[ \norm{\phi_{j,\varepsilon}}_{L^{\infty}(S)} \leq C. \] 
\end{prop}

\begin{proof}
Let \( f(s) := Cap_{\theta}( \phi_{j,\varepsilon} < - s )^{1/n} \). It is clear that \( f : \Rbb^{+} \to \Rbb^{+} \) is right-continuous, and \( \lim_{s \to + \infty} f(s) = 0 \) (cf. Prop. \ref{capacity_properties}). Moreover, \( f \) is decreasing:  for all \( t > s \), \( \set{\phi < -t} \subset \set{\phi < -s}, \; \forall t > s \), hence \( f(t) \leq f(s) \). Following Lem. \ref{capacity_estimate} and the fact that \( \mu_j \) satisfy \( \Hcal(\alpha,A,\theta) \), \( f \) satisfies the condition \( H( \alpha, B) \) with \( B = A^{1/n} \). Indeed, 
\begin{align*}
t^n f(s+t)^n &\leq t^n Cap_{\theta + \varepsilon \omega_B} ( \phi_{j,\varepsilon} < - s -t) \\
&\leq \int_{\set{ \phi_{j,\varepsilon} < -s}} (\theta + \varepsilon \omega_B + d_B d_B^c \phi_{j,\varepsilon} )^n \wedge \eta \\
&= \int_{\set{ \phi_{j,\varepsilon}  < -s}} \mu_j \leq A Cap_{\theta}( \phi_{j,\varepsilon} < -s)^{1+\alpha} = A f(s)^{n(1+\alpha)}. 
\end{align*}
The first inequality follows from Lem. \ref{capacity_properties}, the second is direct from Lem. \ref{capacity_estimate}, while the fourth is a consequence of Lem. \ref{domination_by_capacity}. 
Now let \( \omega_{\varepsilon} := \theta_X + \varepsilon \omega_B \). For \( \varepsilon \) sufficiently small and \( \delta \) large enough, there exists \( \delta = \delta(S) \geq 1 \) such that \( \omega_{\varepsilon} \leq \delta \omega_B \). In particular, \( \phi_{j,\varepsilon} \in PSH^{-}(S, \xi, \delta \omega_B) \). Again by Lem. \ref{capacity_properties},
\begin{align*} 
f(s)^n 
&\leq  Cap_{\delta \omega_B} ( \phi_{j,\varepsilon} < -s)  \\
&\leq \frac{\delta^n}{s} \tuple{ \int_S (-\phi_{j,\varepsilon}) \omega_B^{n} \wedge \eta  + n \text{vol}_{\omega_B}(S) }.
\end{align*}
But by (\ref{uniform_boundedness}) of Lem. \ref{compacity_properties},
\[   \int_S - \phi_{j,\varepsilon} d \mu_{\omega_B}  \leq -\sup \phi_{j,\varepsilon} + C(\omega_B) = C(\omega_B). \]  
Therefore, \( f(s) \leq (C_1/s^{1/n}) \), where \( C_1 = C_1( \omega_B, \theta_X) \). We can then apply Lem. \ref{vanishing_capacity} to select \( s_0 = s_0(n,\alpha, A, \omega_B, \theta_X) \) as in \cite[Lemma 2.3, Theorem 2.1]{EGZ} such that
\[ Cap_{\theta_X}(\phi_{j,\varepsilon} < -s) = 0, \; \forall s \geq s_0. \]
In particular, \( \mu_j ( \phi_{j,\varepsilon} < -s_0) = 0 \) by Lem. \ref{domination_by_capacity}. Hence \( \phi_{j,\varepsilon} \geq s_0 \) on \( S \), so there exists \( C = C(n, \alpha, A, \omega_B, \theta_X) \) such that
\[ \norm{\phi_{j,\varepsilon}}_{L^{\infty}(S)} \leq C. \]  
\end{proof}
\subsection{Laplacian estimate}
We will need  the transverse version of the Yau-Aubin inequality, obtained by Siu for two cohomologous forms \cite{Siu87}, but the proof can be generalized to any couple of Kähler forms. Let 
\[ \Delta_{\omega'_B} := \Tr_{\omega'_B} d_B d_B^c \] 
be the Laplacian associated to the transverse Kähler form \( \omega_B' \). 
\begin{lem} \label{transversal_yau_aubin_inequality}
For each transverse Kähler form \( \omega_B' \), there exists a constant \( \kappa \) depending only on the transverse bisectional curvature of \( \omega_B \) such that
\[ \Delta_{\omega'_B} \log \Tr_{\omega_B} \omega_B' \geq - \kappa \Tr_{\omega_B'} \omega_B - \frac{\Tr_{\omega_B} \text{Ric} (\omega_B')}{\Tr_{\omega_B} \omega_B' }, \]
where \( \text{Ric} (\omega'_B) \) is the transverse Ricci curvature. 
\end{lem}

\begin{proof}
On each foliation chart, the transverse Kähler forms depend only on the \( z \)-coordinates. The inequality thus follows from the purely local proof in the compact Kähler case. The reader may consult Appendix \ref{appendice_yau_aubin_transverse} for a proof. 
\end{proof}

The following proposition gives a \textit{a priori} Laplacian estimate of the solution \( \phi_{j, \varepsilon} \) of equation (\ref{transversal_CY_eqn_perturbed}). We follow the arguments of \cite[Appendix B]{BBEGZ}. 
In this section, by a \textit{uniform constant}, we mean a constant independent of the \( j, \varepsilon \) parameters.

\begin{prop} \label{laplacian_estimate}
Let \( \psi := \Phi_B - r_{\xi}^2 \) and \( 
\omega_{\varepsilon} = \theta _X + \varepsilon \omega_B, \; \omega'_{\varepsilon} := \omega_{ \varepsilon} + d_B d_B^c \phi_{j, \varepsilon} \). 
There exist uniform constants \( C_1, C_2 \) such that
\[ \sup_{S \cap \Ucal } \Tr_{\omega_{\varepsilon}} \omega'_{\varepsilon}  \leq C_2 e^{-C_1 \psi - \psi_{-,j}} \leq C_2 e^{-C_1 \psi - \psi_{-}}. \] 
In particular, there exists a uniform constant \( C_3 \) such that
\[ \sup_{ S \cap \Ucal } \abs{\Delta_{\omega_B} \phi_{j,\varepsilon}} \leq C_3 e^{-C_1 \psi - \psi_{-}}. \]
 
\end{prop}

\begin{proof}
The function \( \psi \) is clearly basic \( \theta_X \)-psh and \( \psi \to -\infty \) near \( \del  \Ucal \) by the construction of \( \Phi_B \) in Lem. \ref{basic_form_asymptotic}. Moreover, \( \omega_B |_{\Ucal} = (\theta_X + dd^c \psi) |_{\Ucal} \) is the restriction into \( \Ucal \) of the transverse Kähler form \( \omega_B \), constructed on \( X \backslash E_0 \). 

Consider the following smooth function on \( S \cap \Ucal \): 
\[ h := \log ( \Tr_{\omega_{\varepsilon}} \omega_{\varepsilon}' ) + n \psi_{-,j} - A_1 ( \phi_{j, \varepsilon} - \psi), \]
where \( A_1 := A_1(\kappa)  \) is a constant sufficiently large and depends on \( \kappa \). The compactness of \( S \), the \( L^\infty \)-estimate in Prop. \ref{linfty_estimate_pluripotential}, combined with transverse Yau-Aubin inequality in Lem. \ref{transversal_yau_aubin_inequality} are all the ingredients we need to repeat the arguments of \cite[Appendix B]{BBEGZ} to conclude. 

For the reader's convenience, we provide here some details of the proof. By the transverse Yau-Aubin inequality, we have on \( S \cap \Ucal \): 
\[ \Delta_{\omega'_{\varepsilon}} h  \geq \Tr_{\omega'_{\varepsilon}}(\omega_{\varepsilon}) - A_2, \] 	
where \( A_2 \) depends only on \( A_1 \) and \( n \). Since \( \phi_{j,\varepsilon} \) is uniformly bounded and that \( \psi  \to -\infty \) near \( \del (S \cap \Ucal) \), \( h \) attains its maximum at \(x_0 \in S \cap \Ucal \). It follows from the maximum principle that
\[ 0 \geq \Delta_{\omega'_{\varepsilon}} h(x_0)  \geq \Tr_{\omega'_{\varepsilon}}(\omega_{\varepsilon})(x_0) - A_2. \]
By local elementary reasonings as in the compact Kähler case, we obtain the following inequality for two transverse Kähler forms: 
\[ \Tr_{\omega_{\varepsilon}}(\omega'_{\varepsilon}) \leq  n \frac{(\omega'_{\varepsilon})^{n}}{\omega_{\varepsilon}^{n}} (\Tr_{\omega'_{\varepsilon}}(\omega_{\varepsilon}))^{n} = (n+1) e^{\psi_{+,j} - \psi_{-,j}} (\Tr_{\omega'_{\varepsilon}}(\omega_{\varepsilon}))^{n}.  \]
Taking log on both sides gives us 
\[ \log (\Tr_{\omega_{\varepsilon}}\omega'_{\varepsilon}) \leq \log(n) + (n+1)(\psi_{+,j} - \psi_{-,j}) + n \log ( \Tr_{\omega'_{\varepsilon}} \omega_{\varepsilon} ), \]  
hence by definition of \( h\), 
\[ h \leq \log(n) + (n + 1) \psi_{+,j} +  n \log( \Tr_{\omega'_{\varepsilon}} \omega_{\varepsilon} ) - A_1(\phi_{j,\varepsilon} - \psi). \] 
Therefore
\[ \sup_{S \cap \Ucal} h \leq h(x_0) \leq A_3 - A_1 \inf_{S \cap \Ucal } (\phi_{j,\varepsilon} - \psi) \leq A_3 - A_1 \inf_{S \cap \Ucal } \phi_{j,\varepsilon},  \] 
where \( A_3 \) is a uniform constant since \( \psi_{+,j} \) and \( \Tr_{\omega'_{\varepsilon}} \omega_{\varepsilon}(x_0) \) are both uniformly bounded.  
As a consequence, there exists a uniform constant \( A_4 \) such that 
\[ h := \log ( \Tr_{\omega_{\varepsilon}} \omega_{\varepsilon}' ) + (n+1) \psi_{-,j} - A_1 ( \phi_{j, \varepsilon} - \psi) \leq A_4, \]
which leads to
\[ \Tr_{\omega_{\varepsilon}} \omega_{\varepsilon}' \leq e^{-(n+1)\psi_{-,j}} e^{A_1(\phi_{j,\varepsilon} - \psi)} e^{A_4}, \] 
hence the existence of uniform constants \( A_1, A_5 \), depending only on \( C \) in inequality (\ref{regularization}),  \( \kappa \), and the bound of the \( L^{\infty} \)-estimate \ref{linfty_estimate_pluripotential} such that 
\[ \sup_{S \cap \Ucal } \Tr_{\omega_{\varepsilon}} \omega'_{\varepsilon}  \leq A_5  e^{-A_1 \psi - \psi_{-,j}} \leq A_4 e^{-A_1 \psi - \psi_{-}}. \] 
For the estimate of \( \Delta_{\omega_B} \phi_{j, \varepsilon} \), we make the following remark. By compactness of \( S \), there exists a uniform constant \( \delta \) sufficiently large such that  
\[ \omega_{\varepsilon} = \theta + \varepsilon \omega_B \leq \delta \omega_B, \]
hence
\[ \Tr_{\omega_B}(.) \leq \delta^{-1} \Tr_{\omega_{\varepsilon}}(.).\]  
But since
\[ \sup_{S \cap \Ucal } \Tr_{\omega_{\varepsilon} } ( \omega_{\varepsilon}  + dd^c \phi_{j, \varepsilon} ) =  n + \sup_{S \cap \Ucal } \Delta_{\omega_{ \varepsilon}} \phi_{j, \varepsilon}  \leq A_4 e^{-A_1 \psi - \psi_{-}}, \] 
this completes our proof. 
\end{proof}

\subsection{Conclusion}
\begin{proof}[Proof of the main Theorem]
By using the \( L^{\infty} \)-estimate in Lem. \ref{linfty_estimate_pluripotential} and the transverse Yau-Aubin inequality \ref{transversal_yau_aubin_inequality}, we obtained in Lem. \ref{laplacian_estimate} the estimate of \( \Delta_{\omega_B}  \phi_{j, \varepsilon} \). 
As a consequence, \( \Delta_{\omega_B} \phi_{j, \varepsilon} \) is locally uniformly bounded on  \( S \cap \Ucal \) since \( \psi_{-} := \Psi_{-} + \phi_X \) is locally bounded by our assumption. It follows that there exists a subsequence \( \phi_{j, \varepsilon(j)} \) which is \( C^1 \)-convergent on \( S \cap \Ucal \) to 
\[ \phi_0 \in L^{\infty}(S \cap \Ucal), \Delta_{\omega_B}  \phi_0 \in L^{\infty}_{\text{loc}}(S \cap \Ucal), \]
which is a solution of
\begin{equation} \label{equation_sasaki_MA}
(\theta + d_Bd_B^c \phi_0)^{n} \wedge \eta = e^{-(n+1) \phi_X } \pi^{*} dV_Y(-J\xi,.)
\end{equation}
on \( S \cap \Ucal \). The equation admits a unique solution up to a constant (cf. Prop. \ref{uniqueness}), hence  
\[ \phi_0 = \phi_X + c, \]
which implies that \( \Delta_{\omega_B} \phi_X \) is locally bounded. This allows us to obtain a \( C^{2,\alpha} \)-estimate of \( \phi_X \), as well as higher order estimates using Schauder's estimate and complex Evans-Krylov theory as in \cite[5.3, p.210]{Blo05}, hence the smoothness of \( \phi_X \) on \( S \cap \Ucal \). 

By definition, 
\( r^2 = r_{\xi}^2 e^{\phi} \) and \( \phi_X = \pi^{*} \phi \). Using symmetry by \( \Rbb_{ > 0} \)-action generated by \( -J\xi \), we conclude that \( \phi_X = \phi \circ \pi \) is actually smooth on \( \Ucal \), hence \( \phi \) is smooth on \( Y_{\text{reg}} \). In particular, \( r^2 \) is smooth on \( Y_{\text{reg}} \).  
\end{proof}

\subsection{Further discussions}
It is expected that the solution of (\ref{equation_sasaki_MA}) is globally continuous over the manifold \( S \). This might be proved by interpolating the viscosity method for Monge-Ampère equations as developed in \cite{EGZ11} into the (degenerate) Sasakian context, using foliation charts on \( S \). 

Most of the pluripotential picture in section 2 can be enlarged to cover the case where \( \theta \) is merely a big transverse form. Indeed, based on \cite{BEGZ10}, one can still make sense of basic \( \theta\)-psh functions, capacities, and the Sasakian Monge-Ampère measure as a non-pluripolar product, again using foliation charts. It is then possible to repeat the proof of the volume-capacity comparison property and deduce the \( L^{\infty} \)-estimate in the big case. As in \textit{loc. cit.}, we expect that the Laplacian estimate holds for big and nef \( \theta \), leading to the smoothness over the ample locus of \( \theta \). This should be true for a big class in general, although the author is unaware of any available techniques. 

\section{Appendix : Transverse Yau-Aubin inequality} \label{appendice_yau_aubin_transverse}

In the sequel, we will use the summation convention. Let \( \omega_B , \omega_B' \) be two transverse Kähler forms on \( S \). Let \( (z,x) \) be the coordinates on a foliation chart of \( S \) such that
\[ \omega_B = g_{ j \ol{k}} \sqrt{-1} dz^{j} \wedge  d \ol{z}^k, \quad \omega'_B = g'_{ j \ol{k}} \sqrt{-1} dz^{j} \wedge  d \ol{z}^k. \] 
After choosing a normal transverse holomorphic chart, one can suppose that \( g_{j \ol{k}} = \delta_{jk} \) and that \( \omega'_B \) is diagonal. Let \( ( g^{j \ol{k}}) \) denote the inverse of \( ( g_{j \ol{k}} ) \). We have  
\[ \Tr_{\omega_B} \omega_B' = g^{j \ol{j}} g'_{j \ol{j}} = \sum_{j}g'_{j \ol{j}}, \quad \Tr_{\omega_B'} \omega_B = g'^{j \ol{j}} g_{j \ol{j}} = \sum_{j} g'^{j \ol{j}}. \]
Denote by
\[ \del_j := \frac{\del}{\del z_j}, \; \delb_k := \frac{\del}{\del \ol{z}_k}, \;  \del_j \delb_k := \frac{\del^2}{ \del z_j \del \ol{z}_k}. \] 

\begin{lem} We have the following inequality:
 \[ g'^{p \ol{p}} (\del_p g'_{a \ol{a}}) (\delb_p g'_{b \ol{b}}) \leq (\Tr_{\omega_B} \omega_B')  \sum_{p,a,j} g'^{p \ol{p}} g'^{a \ol{a}} \abs{\del_p g'_{a \ol{j}} }^2. \] 
\end{lem}

\begin{proof}
The lemma follows from repeated applications of Cauchy-Schwarz inequality:
\begin{align*}
\sum_{p,a,b} g^{p \ol{p}} (\del_p g'_{a \ol{a}})(\delb_p g'_{b \ol{b}}) & \leq \sum_{a,b} ( g^{p \ol{p}} \abs{\del_p g'_{a \ol{a}} }^2 )^{1/2} ( g^{p \ol{p}} \abs{ \delb_p g'_{b \ol{b}}}^2 ) ^{1/2} \\
&= ( \sum_{a} (\sum_p g^{p \ol{p}} \abs{\del_p g'_{a \ol{a}}}^2 )^{1/2} )^2 \\
&= ( \sum_a \sqrt{g_{a \ol{a}}'} ( \sum_p g^{p \ol{p}}g'^{a \ol{a}} \abs{ \del_p g'_{a \ol{a}} }^2 )^{1/2} )^2 \\
&\leq (\sum_{a} g'_{a \ol{a}})( \sum_{p,a} g^{p \ol{p}} g'^{a \ol{a}} \abs{\del_p g'_{a \ol{a}}}^2 ) \\
& \leq (\Tr_{\omega_B} \omega_B') ( \sum_{p,a,j} g^{p \ol{p}} g'^{a \ol{a}} \abs{\del_p g'_{a \ol{j}}}^2  ).
\end{align*}
\end{proof}

Recall the statement of the transverse Yau-Aubin inequality:
\begin{lem}
\[ \Delta_{\omega'_B} \log \Tr_{\omega_B} \omega_B' \geq - \kappa \Tr_{\omega_B'} \omega_B - \frac{\Tr_{\omega_B} \text{Ric} (\omega_B')}{\Tr_{\omega_B} \omega_B' }.\] 
\end{lem}

\begin{proof}
We have
\begin{align*}
\Delta_{\omega'_B} \log \Tr_{\omega_B} \omega'_B &=  \frac{\Delta_{\omega'_B} \Tr_{\omega_B} \omega_B'} {\Tr_{\omega_B} \omega'_B} - g^{p \ol{q}} \frac{(\delb_q \Tr_{\omega_B} \omega'_B) (\del_p \Tr_{\omega_B} \omega'_B) }{ (\Tr_{\omega_B} \omega_B')^2 } \\
&= \frac{\Delta_{\omega'_B} \Tr_{\omega_B} \omega_B'} {\Tr_{\omega_B} \omega'_B} - \frac{ g^{p \ol{p}} (\del_p g'_{a \ol{a}}) (\delb_p g'_{b \ol{b}})}{ (\Tr_{\omega_B} \omega_B')^2}.
\end{align*}
By definition,
\begin{align*}
\Delta_{\omega'_B} \Tr_{\omega_B} \omega_B' &= g'^{p \ol{q}} ( \del_p \delb_q  g^{j \ol{k}} ) g'_{j \ol{k}} + g'^{p \ol{q}} g^{j \ol{k}} \del_p \delb_q g'_{j \ol{k}} \\
&=  g'^{p \ol{q}} ( \del_p \delb_q  g^{j \ol{k}} ) g'_{j \ol{k}} - g'^{p \ol{q}} g^{j \ol{k}} R'_{j \ol{k} p \ol{q}} + g'^{p \ol{q}} g^{j \ol{k}} g'^{a \ol{b}} ( \del_p g'_{j \ol{b}} )( \delb_q g'_{a \ol{k}}).
\end{align*}
where \( R'_{j \ol{k} p \ol{q}}  \) is the local expression of the transverse curvature form of \( \omega_B' \). Let us estimate the three terms of the expression above. 
\begin{itemize}
\item Since \( \omega_B \) and \( \omega_B' \) are diagonal, we have for the first term:
\[ g'^{p \ol{q}} ( \del_p \delb_q  g^{j \ol{k}} ) g'_{j \ol{k}} = g'^{p \ol{p}} ( \del_p \delb_p  g^{j \ol{j}} ) g'_{j \ol{j}} \geq - \kappa (\Tr_{\omega_B} {\omega'_B}) (\Tr_{\omega_B'} \omega_B), \]
where \( \kappa \) is the infimum of the transverse sectional curvature (which exists since \( S \) is compact). 
\item In the second term, \( g'^{p \ol{q}} R_{j \ol{k} p \ol{q}} = R'_{j \ol{k}} \), where \(R'_{j \ol{k}} \) is the local expression of the transverse Ricci-form \( \text{Ric}(\omega'_B) \). 
\item 
For the third term, we have:
\[  g'^{p \ol{q}} g^{j \ol{k}} g'^{a \ol{b}} ( \del_p g'_{j \ol{b}} )( \delb_q g'_{a \ol{k}}) = g'^{p \ol{p}} g'^{a \ol{a}} \abs{\del_p g'_{a \ol{j}} }^2. \] 
\end{itemize}
It follows that
\[ \Delta_{\omega'_B} \Tr_{\omega_B} \omega_B' \geq - \kappa \Tr_{\omega_B} {\omega'_B} \Tr_{\omega_B'} \omega_B - g^{j \ol{k}} R'_{j \ol{k}} + \sum_{p, a, j} g'^{p \ol{p}} g'^{a \ol{a}} \abs{\del_p g'_{a \ol{j}} }^2,  \]
hence
\begin{align*}
\Delta_{\omega'_B} \log \Tr_{\omega_B} \omega'_B & \geq - \kappa \Tr_{\omega'_B} \omega_B - \frac{\Tr_{\omega_B} \text{Ric} (\omega_B') } {\Tr_{\omega_B} \omega_B'} \\
&+ \frac{\sum_{p,a,j} g'^{p \ol{p}} g'^{a \ol{a}} \abs{\del_p g'_{a \ol{j}} }^2}{\Tr_{\omega_B} \omega'_B} - \frac{g^{p \ol{p}} (\del_p g'_{a \ol{a}}) (\delb_p g'_{b \ol{b}})}{(\Tr_{\omega_B} \omega_B)^2} \\
& \geq - \kappa \Tr_{\omega'_B} \omega_B - \frac{\Tr_{\omega_B} (\text{Ric} (\omega_B') } {\Tr_{\omega_B} \omega_B'},
\end{align*}
by the previous lemma. 
\end{proof}

\subsection*{Acknowledgements} This article is part of a thesis supervised by Thibaut Delcroix and Marc Herzlich, partially supported by ANR-21-CE40-0011 JCJC project MARGE. I wish to thank Vincent Guedj, Eleonora Di Nezza, and Tat-Dat To for their generosity as well as many helpful discussions and remarks. Thanks are also due to the hospitality of the Vietnam Institute for Advanced Study in Mathematics (VIASM), where this work first begun.

\bibliographystyle{alpha}
\bibliography{biblio}

\end{document}